\documentclass[10pt, twoside]{article}
\usepackage{amsmath,amsthm,amssymb}
\usepackage{times}
\usepackage{enumerate}
\usepackage{amssymb}
\usepackage{graphicx}
\usepackage[noadjust]{cite}

\usepackage{color}

\def\titlerunning#1{\gdef\titrun{#1}}
\makeatletter
\def\author#1{\gdef\autrun{\def\and{\unskip, }#1}\gdef\@author{#1}}

\makeatother

\def\keywords#1{\par\medskip
\noindent\textbf{Keywords.} #1}
\def\subjclass#1{\par\smallskip
\noindent\textbf{MSC (2010):} #1}

\newtheorem{thm}{Theorem}[section]

\newtheorem{lem}[thm]{Lemma}

\newtheorem{prop}[thm]{Proposition}

\theoremstyle{definition}
\newtheorem{defin}[thm]{Definition}
\newtheorem{rem}[thm]{Remark}
\newtheorem{exa}[thm]{Example}

\numberwithin{equation}{section}

\newtheorem*{notations}{Notations}

\DeclareMathOperator*{\esssup}{ess\,sup}

\makeatletter
\let\@fnsymbol\@alph
\makeatother

\begin{document}

\baselineskip=17pt

\titlerunning{Radial quasilinear problems}

\title{Compactness and existence results for quasilinear elliptic problems with singular or vanishing potentials}

\author{Marino Badiale\thanks{Dipartimento di Matematica ``Giuseppe Peano'', Universit\`{a} degli Studi di
Torino, Via Carlo Alberto 10, 10123 Torino, Italy. 
e-mails: \texttt{marino.badiale@unito.it}, \texttt{michela.guida@unito.it}}
\textsuperscript{,}\thanks{Partially supported by the PRIN2012 grant ``Aspetti variazionali e
perturbativi nei problemi differenziali nonlineari''.}
\ -\ Michela Guida\textsuperscript{a}
\ -\ Sergio Rolando\thanks{Dipartimento di Matematica e Applicazioni, Universit\`{a} di Milano-Bicocca,
Via Roberto Cozzi 53, 20125 Milano, Italy. e-mail: \texttt{sergio.rolando@unito.it}}%
}

\date{
}
\maketitle

\begin{abstract}
Given $\geq 3$, $1<p<N$, two measurable functions $V\left(r \right)\geq 0$, $K\left(r\right)> 0$, and a continuous function $A(r) >0$ ($r>0$), we study the quasilinear elliptic equation
\[
-\mathrm{div}\left(A(|x| )|\nabla u|^{p-2} \nabla u\right) u+V\left( \left| x\right| \right) |u|^{p-2}u= K(|x|) f(u) \quad \text{in }\mathbb{R}^{N}. 
\]
We find existence of nonegative solutions by the application of variational methods, for which we have to study the compactness of the embedding of a suitable
function space $X$ into the sum of Lebesgue spaces $L_{K}^{q_{1}}+L_{K}^{q_{2}}$, and thus into $L_{K}^{q}$ ($=L_{K}^{q}+L_{K}^{q}$) as a particular case. 
Our results do not require any compatibility between how the potentials $A$, $V$ and $K$ behave at the origin and at infinity, and 
essentially rely on power type estimates of the relative growth of $V$ and $K$, not of the potentials separately. The nonlinearity $f$ has a double-power behavior, whose standard example is $f(t) = \min \{ t^{q_1 -1}, t^{q_2 -1}  \}$, recovering the usual case of a single-power behavior when $q_1 = q_2$.

\end{abstract}

\section{Introduction}

In this paper we pursue the work we made in papers \cite{BGR_I,BGR_II, BGR_p, BZ, GR-nls,BGR_bilap}, 
where we studied embedding and compactness results for weighted Sobolev spaces. These results then made possible to get existence and multiplicity results, by variational methods, for several kinds of elliptic equations in $\mathbb{R}^N$.

In the present paper we face quasilinear elliptic equations in presence of a radial potential on the derivatives, that is, equations of the following kind
\begin{equation}\label{EQ}
-\mathrm{div}\left(A(|x|) |\nabla u|^{p-2} \nabla u\right) u+V\left( \left| x\right| \right) |u|^{p-2}u= K(|x|) f(u) \quad \text{in }\mathbb{R}^{N}
\end{equation}
where $1<p<N$, $f:\mathbb{R}\rightarrow \mathbb{R}$ is a continuous nonlinearity satisfying $f\left( 0\right) =0$, and $V, A, K$ are given potentials. We study such equation by variational methods, so we introduce a suitable functional space $X$ (see section 2) and we say that $u\in X$ is a \textit{weak solution}\emph{\ }to (\ref{EQ}) if 
\begin{equation}
\int_{\mathbb{R}^{N}}|\nabla u|^{p-2}\nabla u\cdot \nabla h\,dx+\int_{\mathbb{R}^{N}}V\left(
\left| x\right| \right) |u|^{p-2}uh\,dx=\int_{\mathbb{R}^{N}}K\left( \left| x\right|
\right) f\left( u\right) h\,dx\quad \textrm{for all }h\in X.
\end{equation}
\noindent These solutions are (at least formally) critical points of the Euler
functional 
\begin{equation}
I\left( u\right) :=\frac{1}{p}\left\| u\right\| ^{p}-\int_{\mathbb{R}
^{N}}K\left( \left| x\right| \right) F\left( u\right) dx,  \label{I:=}
\end{equation}
where $F\left( t\right) :=\int_{0}^{t}f\left( s\right) ds$ and $|| \cdot ||$ is the norm on $X$ (see section 2 below). Then the problem
of existence is easily solved if $A\equiv 1$, $V$ does not vanish at infinity, $K$ is
bounded and $f(t)= t^{q-1}$, because standard embeddings theorems for $X$ are available (for suitable $q$'s). As we let $V$ and $K$ to vanish, 
or to go to infinity, as $|x| \rightarrow 0$ or $|x|\rightarrow +\infty$, and we introduce the potential $A$ on the derivatives, the usual embeddings theorems for Sobolev spaces are not 
available anymore, and new embedding theorems need to be proved. This has been done in several papers: see e.g. the references in 
\cite{BGR_I,BGR_II,GR-nls} for a bibliography concerning the usual Laplace equation, \cite{Anoop,Su12,Cai-Su-Sun,SuTian12,Su-Wang-Will-p,Yang-Zhang,Zhang13,BPR}
for equations involving the $p$-laplacian, and \cite{BZ, Su-Wang} and the references therein for problems with a potential $A$ on the derivatives.
 
The main novelty of our approach (in \cite{BGR_I,BGR_II, BGR_p, BZ} and in the present paper) is two-folded. Firstly, we look for embeddings of $X$ not into a single Lebesgue space $L_{K}^{q}$ 
but into a sum of Lebesgue spaces $L_{K}^{q_1}+L_{K}^{q_2}$. This allows to study separately the behaviour of the potentials $V,K$ at $0$ and $\infty$, and to assume independent set of hypotheses about these behaviours. Secondly, we assume hypotheses not on $V$ and $K$ separately but on their ratio, so allowing asymptotic behaviors of general kind for the two potentials.

Thanks to this second novelty we obtain embedding results, and thus existence results for equation \eqref{EQ}, for more general kinds of potentials than the power type ones (cf. Example \ref{ex2} below), which are essentially the only ones considered in the existing literature (cf. \cite{Su-Wang}).
Moreover, thanks to the first novelty, we get new results also for power type potentials (cf. Example \ref{ex2} below).

This paper is organized as follows. In Section 2 we introduce the hypotheses on $A,V,K$ and the function spaces $D_A$ and $X$ in which we will work. 
In Section \ref{COMP} we state a general result concerning the embedding properties of $X$  into $L_{K}^{q_{1}}+L_{K}^{q_{2}}$ (Theorem \ref{THM(cpt)})
and some explicit conditions ensuring that the embedding is compact
(Theorems \ref{THM0} and \ref{THM1}). The general
result is proved in Section \ref{SEC:1}, the explicit conditions in Section 
\ref{SEC:2}. In Section \ref{SEC: ex} we apply our embedding results to get existence of non negative solutions for (\ref{EQ}).
In section \ref{SEC:EX} we give some examples to explain the novelty of our results.

\begin{notations}

We end this introductory section by collecting
some notations used in the paper.

\noindent  $\bullet $ $\mathbb{R}_{+} = ( 0, +\infty ) = \left\{ x\in \mathbb{R} : x>0 \right\}$.

\noindent $\bullet $ For every $R>0$, we set $B_{R} =\left\{ x\in \mathbb{R}
^{N}:\left| x\right| <r\right\} $.

\noindent $\bullet$ $\omega_N$ is the $(N-1)-$dimensional measure of the unit sphere $\partial B_1 = \left\{ x\in \mathbb{R}
^{N}:\left| x\right| =1 \right\} $.

\noindent $\bullet $ For any subset $A\subseteq \mathbb{R}^{N}$, we denote $
A^{c}:=\mathbb{R}^{N}\setminus A$. If $A$ is Lebesgue measurable, $\left|
A\right| $ stands for its measure.



\noindent $\bullet $ $C_{\mathrm{c}}^{\infty }(\Omega )$ is the space of the
infinitely differentiable real functions with compact support in the open
set $\Omega \subseteq \mathbb{R}^{N}$. If $\Omega$ has radial symmetry, $C_{\mathrm{c}, r}^{\infty }( \Omega )$ is the subspace 
of $C_{\mathrm{c}}^{\infty }(\Omega )$ made of radial functions.

\noindent $\bullet $ For any measurable set $A\subseteq \mathbb{R}^{N}$, $
L^{q}(A)$ and $L_{\mathrm{loc}}^{q}(A)$ are the usual real Lebesgue spaces.
If $\rho :A\rightarrow \mathbb{R}_{+}$ is a measurable function, then $%
L^{p}(A,\rho \left( z\right) dz)$ is the real Lebesgue space with respect to
the measure $\rho \left( z\right) dz$ ($dz$ stands for the Lebesgue measure
on $\mathbb{R}^{N}$). In particular, if $K:\mathbb{R}_{+}\rightarrow \mathbb{R}
_{+} $ is measurable, we denote $L_{K}^{q}\left( A\right) :=L^{q}\left(
A,K\left( \left| x\right| \right) dx\right) $.

\noindent $\bullet $ $p^{\prime }:=p/(p-1)$ is the H\"{o}lder-conjugate
exponent of $p.$

\end{notations}

\section{Hypotheses and preliminary results  }

\par \noindent Throughout this paper we assume $N\geq 3$ and 
$1<p<N$. We will make use of the following hypotheses on $A,V,K$:

\begin{itemize}

\item[$\left( \mathbf{A}\right) $]  $A : \mathbb{R}_{+} \rightarrow \mathbb{R}_{+} $ is continuous and there exist real numbers $p-N<a_{0},a_{\infty} \leq p$ and $c_{0},c_{\infty}>0$ such that:
$$
c_{0} \leq \liminf_{r\rightarrow 0^+}\frac{A(r)}{r^{a_{0}}} \leq \limsup_{r\rightarrow 0^+}\frac{A(r)}{r^{a_{0}}} < + \infty, 
$$
$$
c_{\infty}\leq \liminf_{r\rightarrow +\infty}\frac{A(r)}{r^{a_{\infty}}} \leq \limsup_{r\rightarrow +\infty}\frac{A(r)}{r^{a_{\infty}}} <+\infty .
$$

\item[$\left( \mathbf{V}\right) $]  $V:\mathbb{R}_{+}\rightarrow
\left[ 0,+\infty \right) $ is measurable and such that $V\in
L_{\mathrm{loc}}^{1}\left( \mathbb{R}_{+} \right) ;$

\item[$\left( \mathbf{K}\right) $]  $K:\mathbb{R}_{+} \rightarrow \mathbb{R}_{+} $ is measurable and such that 
$K\in L_{\mathrm{loc}}^{s}\left( \mathbb{R}_{+} \right) $ for some $s>1$.

\end{itemize}

\begin{rem}
\label{remA}
It is easy to check that the hypothesis $\left( \mathbf{A}\right) $ implies that, for each $R>0$, there exist $C_{0}=C_{0}(R)>0$ and $C_{\infty}=C_{\infty}(R)>0$ such that 
\begin{equation}\label{remA.1}
A(|x|)\geq C_{0}|x|^{a_{0}} \quad   \mbox{for all} \,\, 0<|x|\leq R,
\end{equation}
\begin{equation}\label{remA.2}
A(|x|)\geq C_{\infty}|x|^{a_{\infty}} \quad  \mbox{for all} \,\, |x|\geq R.
\end{equation}
\end{rem}

We now introduce the space functions in which we will work. These are the following:

\begin{itemize}

\item  $D_A$ is the closure of $C_{\mathrm{c}, r}^{\infty }( \mathbb{R}^{N} ) $ with respect to the norm $||u||_A := \left( \int_{\mathbb{R}^{N}} A(|x|)|\nabla u|^p 
dx \right)^{1/p}$ (see also Definition \ref{DEF:D_A} below),

\item $X := D_A \cap  L^{p}(\mathbb{R}^{N}, V (|x|) dx )$ with norm $||u|| := \left( ||u||_{A}^{p} + ||u||_{L^{p}(\mathbb{R}^{N}, V (|x|) dx )}^{p} \right)^{1/p}$.

\end{itemize}

\noindent The rest of this section is devoted to elucidate the characteristics of the functions in $D_A$. In particular we prove some relevant pointwise estimates and embedding results. To be precise, we define 

$$
S_A := \left\{  u \in C_{\mathrm{c}, r}^{\infty }( \mathbb{R}^{N} )  \, \Big| \, \int_{\mathbb{R}^{N}}A(|x|) |\nabla u |^p dx <+\infty  \right\}.
$$

\noindent $S_A$ is a linear subspace of $C_{\mathrm{c}, r}^{\infty }( \mathbb{R}^{N} ) $ and 
$||u||_A = \left( \int_{\mathbb{R}^{N}}A(|x|) |\nabla u |^p dx \right)^{1/p}$ is a norm on it. The next lemmas gives the relevant pointwise estimates for the functions in $S_A$. In all this paper, for any radial function $u$, with a little abuse of notations, we will write $u(x)= u(|x|)= u(r)$ if $r =|x|$.

\begin{lem}
\label{A1}
Assume the hypothesis $\left( \mathbf{A}\right) $. Fix $R_0>0$. Then there exists a constant $C=C(N, R_0, p, a_{\infty} )>0$ such that for all $u \in S_A$ one has
\begin{equation}
\label{est1}
|u(x)|\leq C\, |x|^{-\frac{N+a_{\infty}-p}{p}}\, \left( \int_{B^{c}_{R_0}} A(|x|) \, |\nabla u |^p \, dx \right)^{1/p}    \quad \text{for } \;|x|\geq R_0.
\end{equation}
\end{lem}

\begin{proof}
Assume $u\in S_A$. For $|x|= r\geq R_0$ we have
\begin{equation*}
-u(r)=\int_{r}^{\infty}u'(s)ds.
\end{equation*}
Using the hypothesis $\left( \mathbf{A}\right) $ and H\"older inequality, we obtain
\begin{eqnarray*}
|u(r)| &\leq& \int_{r}^{\infty}|u'(s)|ds 
=\int_{r}^{\infty}|u'(s)|s^{\frac{N+a_{\infty}-1}{p}}s^{-\frac{N+a_{\infty}-1}{p}}ds
\\
&\leq&\left(\int_{r}^{\infty}|u'(s)|^{p}s^{N-1}s^{a_{\infty}}ds\right)^{\frac{1}{p}}\left(\int_{r}^{\infty}s^{-\frac{{N+a_{\infty}-1}}{p-1}} ds\right)^{\frac{p-1}{p}}
\\
&=&  (\omega_{N})^{-\frac{1}{p}}\left(\int_{B_{r}^c}|x|^{a_{\infty}}|\nabla u|^{p}dx\right)^{\frac{1}{p}}\left(\int_{r}^{\infty} s^{-\frac{{N+a_{\infty}-1}}{p-1}} ds\right)^{\frac{p-1}{p}}
\\
&\leq&  C_{\infty} ^{-1/p} \omega_{N}^{-\frac{1}{p}}\left( \frac{p-1}{a_{\infty} +N-p}    \right)^{\frac{p-1}{p}}  \left(\int_{B_{r}^c} A(|x|)|\nabla u|^{p}dx\right)^{\frac{1}{p}}\, r^{-\frac{{N+a_{\infty}-p}}{p}},
\end{eqnarray*}
where $C_{\infty}=C_{\infty}(R_0)$ is the constant in \ref{remA.2}. Hence the thesis follows. 
\end{proof}

\begin{lem}
\label{A2} 
Assume the hypothesis $\left( \mathbf{A}\right) $. Fix $R_0 >0$. Then there exists a constant $C= C(N, R_0, p, a_0 )>0$ such that for all $u\in S_A \cap C_{c,r}^{\infty }(B_{R_0})$ one has
\begin{equation}
|u(x)|\leq C\, |x|^{-\frac{N+a_{0}-p}{p}}\, \left( \int_{B_{R_0}} A(|x|) \, |\nabla u |^p \, dx \right)^{1/p} 
\quad \text{for }0<|x|<R_0.
\end{equation}
\end{lem}

\begin{proof}
Let $u\in S_A \cap C_{c,r}^{\infty }(B_{R_0})$ and take $|x|=r<R_0$.
Since $u(R_0 )=0$, we have
\begin{equation*}
-u(r)=u(R_0)-u(r)=\int_{r}^{R_0}u'(s)ds.
\end{equation*}
The same arguments of Lemma \ref{A1} yield
\begin{eqnarray*}
|u(r)|&\leq&\int_{r}^{R_0}|u'(s)|ds
\leq\left(\int_{r}^{R_0}|u'(s)|^{p}s^{N-1}s^{a_{0}}ds\right)^{\frac{1}{p}}\left(\int_{r}^{R_0}  s^{-\frac{N+a_{0}-1}{p-1}}  ds\right)^{\frac{p-1}{p}}
\\
&\leq &(\omega_{N})^{-\frac{1}{p}} \left(\int_{B_{R_0} } |x|^{a_0}|\nabla u|^{p}dx\right)^{\frac{1}{p}} \left(\int_{r}^{R_0}s^{-\frac{N+a_{0}-1}{p-1}} ds\right)^{\frac{p-1}{p}}
\\
&\leq & \omega_{N}^{-\frac{1}{p}}  C_0^{-1/p}\, \left(\frac{p-1}{N+a_{0}-p} \right)^{p-1/p} \left( \int_{B_{R_0}} A(|x|) \, |\nabla u |^p \, dx \right)^{1/p} \,   r^{-\frac{N+a_{0}-p}{p}},
\end{eqnarray*}
where $C_0=C_0 (R_0)$ is the constant in \ref{remA.1}. Then the thesis ensues.
\end{proof}

\begin{lem}
\label{ALem}
Assume the hypothesis $\left( \mathbf{A}\right) $. Fix $R_0 >0$. Then there exists a constant $C= C(N, R_0, p,a_0 , a_{\infty})>0$ such that for all $u\in   S_A$ one has
\begin{equation*}
|u(x)|\leq C\, |x|^{-\frac{N+a_{0}-p}{p}}\, \left( \int_{B_{R_0 +1 }} A(|x|) \, |\nabla u |^p \, dx + \int_{B^{c}_{R_0  }} A(|x|) | \nabla u|^p  dx      \right)^{1/p} \quad \text{for } 0<|x|<R_0.
\end{equation*}
\end{lem}

\begin{proof}
Let $u\in   S_A$. Take a radial function $\rho  \in    C_{c,r}^{\infty}(\mathbb{R}^{N})$ such that 
$ \rho (x) \in [0,1]$, $\rho \equiv 1 $ in $B_{R_0} $ and $\rho (x) \equiv 0$ if $|x| \geq R_0 + 1/2$. Hence $\rho \, u \in C_{c,r}^{\infty}(B_{R_0 +1})$, so that we can apply Lemma \ref{A2} (in the ball $B_{R_0 +1} )$ and get
$$ | \rho (x) u (x) | \leq C |x|^{-\frac{N+ a_0 -p}{p}} \left( \int_{B_{R_0 +1}} A(|x|) |\nabla (\rho  u)|^p dx \right)^{1/p}.$$
If $|x| < R_0 $ we have $\rho (x) =1$ and hence
$$ | u (x) | \leq C |x|^{-\frac{N+ a_0 -p}{p}} \left( \int_{B_{R_0 +1}} A(|x|) |\nabla  (\rho u)|^p dx \right)^{1/p}.$$
We also have
$$|\nabla  (\rho u)|^p \leq \left( \rho |\nabla u | + |u| |\nabla \rho | \right)^p \leq C_p \left(\rho^p |\nabla u|^p + |u|^p |\nabla \rho |^p \right)$$
and hence, for $x \in B_{R_0}$,
$$ | u (x) | \leq C |x|^{-\frac{N+ a_0 -p}{p}} C_p^{1/p}\left( \int_{B_{R_0 +1}} A(|x|) |\nabla  u|^p dx + \int_{B_{R_0 +1}} A(|x|) |u|^p \,|\nabla  \rho|^p dx\right)^{1/p}.$$
Moreover
$$\int_{B_{R_0 +1}} A(|x|) |u|^p \,|\nabla  \rho|^p dx \leq C_1 \int_{B_{R_0 +1}\backslash B_{R_0}} A(|x|) |u|^p \, dx$$
where the constant $C_1= \max |\nabla \rho |^p $ depends only on $\rho$, hence on $R_0$. We can now apply Lemma \ref{A1} in the domain $B^{c}_{R_0}$, and we get
$$|u|^p (y) \leq C |y|^{-N- a_{\infty} +p} \int_{B^{c}_{R_0 }} A(|x|) |\nabla  u|^p dx \quad \text{for } |y|> R_0. $$
Recalling that $A$ is bounded in $B_{R_0 +1}\backslash B_{R_0} $, for $y \in B_{R_0 +1}\backslash B_{R_0} $ and $C_2 = 
\max \left\{ A(|y|) \, \big| \, y\in B_{R_0 +1}\backslash B_{R_0} \right\}$ we get
$$A(|y|) |u|^p (y)  \leq C_2 \, C |y|^{-N - a_{\infty} +p} \int_{B^{c}_{R_0 }} A(|x|) |\nabla  u|^p dx$$
and hence, integrating w.r.t. $y \in B_{R_0 +1}\backslash B_{R_0} $, we obtain
$$\int_{B_{R_0 +1}\backslash B_{R_0}}A(|y|) |u|^p (y) dy \leq C_3 \int_{B^{c}_{R_0 }} A(|x|) |\nabla  u|^p dx$$
where $C_3=C_3 (N, a_{\infty}, R_0 , p)$. Pasting all together, for $|x|<R_0$ we get
\begin{eqnarray*}
| u (x) | &\leq &C |x|^{-\frac{N+ a_0 -p}{p}} C_p^{1/p}\left( \int_{B_{R_0 +1}} A(|x|) |\nabla  u|^p dx + \int_{B_{R_0 +1}} A(|x|) |u|^p \,|\nabla  \rho|^p dx\right)^{1/p} 
\\
&\leq &C_4 |x|^{-\frac{N+ a_0 -p}{p}} \left( \int_{B_{R_0 +1}} A(|x|) |\nabla  u|^p dx + \int_{B^{c}_{R_0 }} A(|x|) \,|\nabla  u|^p dx \right)^{1/p}, 
\end{eqnarray*}
where $C_4 =C_4 (N, R_0 , p, a_0, a_{\infty})$. Hence the thesis.
\end{proof}

We can now give a precise definition of $D_A$.

\begin{defin}\label{DEF:D_A}
$D_A$ is the completion of $S_A$ with respect to the norm $|| \cdot ||_A$.
\end{defin}

The pointwise estimates of the previous lemmas imply the following proposition, which gives the main properties of $D_A$. The proof is a simple exercise in funtional analysis and we skip it.

\begin{lem}
	\label{lemdens2}
	
	Assume the hypothesis $\left( \mathbf{A}\right) $. Then the following properties hold.
	
	\begin{itemize}

\item[{$(i)$}] If $u \in D_A$, then $u$ has weak derivatives $D_i u$ in the open set $\Omega=\mathbb{R}^{N} \backslash \{0 \}$ ($i=1,..., N$) and one has $D_i u \in L^p_{loc}(\Omega )$.

\item[{$(ii)$}] If $u \in D_A$, then $\int_{\mathbb{R}^{N}} A(|x|) |\nabla u|^p \, dx < + \infty $ and $||u||_A =\left( \int_{\mathbb{R}^{N}} A(|x|) |\nabla u|^p \, dx
	\right)^{1/p}$ is a norm on $D_A$. With this norm, $D_A$ is a Banach space.

\item[{$(iii)$}] For any $R_0 >0$, there exists a constant $M=M(N, R_0 , a_0 , a_{\infty})>0$ such that for all $u \in D_A$ one has 

$$ |u(x) | \leq M |x|^{-\frac{N+ a_0 -p}{p}} \, ||u||_A , \quad \mbox{for a.e.}\, x \in B_{R_0}, $$

$$ |u(x) | \leq M |x|^{-\frac{N+ a_{\infty} -p}{p}} \, ||u||_A ,\quad \mbox{for a.e.}\,  x \in B^{c}_{R_0}. $$

\end{itemize}

\end{lem}

We now prove some embedding properties of the space $D_A$. To this aim, we define the exponents $p_{0}, p_{\infty}$ as follows: 
\begin{equation*}
p_{0}:=\frac{pN}{N+a_{0}-p}\,,\quad p_{\infty}:=\frac{pN}{N+a_{\infty}-p}.
\end{equation*}
Recall that $p-N<a_{0},\, a_{\infty}\leq p$ and notice that, from the hypotheses, we have $p_{0}, p_{\infty}\geq p$.

\begin{lem}
\label{A4}

Assume the hypothesis $\left( \mathbf{A}\right) $. For every $R >0$ we have the continuous embeddings
\begin{equation*}
D_A \hookrightarrow L^{p_{0}}(B_{R}) \quad \text{and} \quad D_A \hookrightarrow L^{p_{\infty}}(B^{c}_{R}).
\end{equation*}
\end{lem}

\begin{proof}
Let $u\in    S_A$ (the general case follows by density). Fix $R>0$ and denote by $C$ any positive constant only dependent on $N$, $p$, $a_0$, $a_\infty$ and $R$. We first prove the embedding in $L^{p_{\infty}}(B^{c}_{R})$, so we estimate the norm of $u$ in such a space. With an integration by parts, we get

$$\int_{B_{R}^c} |u(x)|^{p_{\infty}} dx = \omega_N \int_{R}^{+\infty} r^{N-1} |u (r)|^{p_{\infty}}dr \leq \frac{\omega_N p_{\infty} }{N}
\int_{R}^{+\infty} r^N |u (r)|^{p_{\infty}-1}\, |u' (r) | dr.
$$

\noindent Then, by several applications of H\"older inequality and using the pointwise estimates of Lemma \ref{lemdens2}, we get 

$$\int_{R}^{+\infty} r^N |u (r)|^{p_{\infty}-1} \, |u' (r) | dr = \int_{R}^{+\infty} r^{\frac{N-1}{p}} |u' (r)| r^{\frac{a_{\infty}}{p}} \,
r^{-\frac{a_{\infty}}{p}} r^{\frac{Np-N+1}{p}} |u (r)|^{p_{\infty}-1} dr $$
$$
\leq \left( \int_{R}^{+\infty} r^{N-1}  |u' (r)|^p r^{a_{\infty}} dr \right)^{1/p} \, \left( \int_{R}^{+\infty} r^{\frac{Np-N+1-a_{\infty}}{p-1}}
|u (r)|^{(p_{\infty}-1)\frac{p}{p-1} }dr \right)^{\frac{p-1}{p}} 
$$

$$
\leq C\left(\int_{R}^{+\infty}  r^{N-1} A(r) |u' (r)|^p   dr  \right)^{1/p} \, \left( \int_{R}^{+\infty} r^{N-1} |u(r)|^{p_{\infty}} \,
r^{\frac{p-a_{\infty}}{p-1}} |u(r)|^{\frac{p_{\infty} -p}{p-1}} dr   \right)^{\frac{p-1}{p}} 
$$

$$
\leq C \left( \int_{B_{R}^c} A(|x|) \, |\nabla u (x)|^p dx  \right)^{1/p} \times$$
$$\times \left(\int_{R}^{+\infty} r^{N-1} |u(r)|^{p_{\infty}} 
r^{\frac{p-a_{\infty}}{p-1} } \, r^{-\frac{N+a_{\infty}-p}{p} \frac{p_{\infty}-p}{p-1} } \, \left(  \int_{B_{R}^c} A(|x|) \, |\nabla u (x)|^p dx \right)^{\frac{p_{\infty}-p}{p(p-1)}}  \right)^{\frac{p-1}{p}}  
$$

$$
\leq C \left( \int_{B_{R}^c} A(|x|) \, |\nabla u (x)|^p dx  \right)^{\frac{p_{\infty}}{p^2}} \, \left( \int_{B_{R}^c} |u(x)|^{p_{\infty}} dx\right)^{\frac{p-1}{p}}
$$

\noindent 
Notice that we have $\frac{p-a_{\infty}}{p-1} - \frac{N+a_{\infty}-p}{p} \frac{p_{\infty}-p}{p-1} =0  $, from the definition of $p_{\infty}$. 

\par\noindent From the previous computations, we get

$$
\int_{B_{R}^c} |u(x)|^{p_{\infty}} dx \leq C \left( \int_{B_{R}^c} A(|x|) \, |\nabla u (x)|^p dx  \right)^{\frac{p_{\infty}}{p^2}} \, \left( \int_{B_{R}^c} |u(x)|^{p_{\infty}} dx\right)^{\frac{p-1}{p}}
$$ 

\noindent and hence

$$
\left( \int_{B_{R}^c} |u(x)|^{p_{\infty}} dx  \right)^{1/p} \leq C \left( \int_{B_{R}^c} A(|x|) \, |\nabla u (x)|^p dx  \right)^{\frac{p_{\infty}}{p^2}} ,
$$

\noindent that is

$$
\left( \int_{B_{R}^c} |u(x)|^{p_{\infty}} dx  \right)^{\frac{1}{p_{\infty}}} \leq C \left( \int_{B_{R}^c} A(|x|) \, |\nabla u (x)|^p dx  \right)^{\frac{1}{p}} \leq
C ||u||_A , 
$$

\noindent which is the thesis. In order to prove the embedding in $L^{p_{0}}(B_{R})$, we use an argument similar to the one of Lemma \ref{ALem}. So we fix a cut-off function $\rho$ as we did there (with $R$ instead of $R_0$). For $u \in S_A$ we set $v= \rho u \in C_{c,r}^{\infty}(B_{R+1})$. Arguing as for the previous case, we get

$$
\int_{B_{R+1}} |v(x)|^{p_0} dx = \omega_N \int_{0}^{R+1} r^{N-1} |v(r)|^{p_0}  dr \leq \frac{p_0 \omega_N}{N} \int_{0}^{R+1} r^{N} |v(r)|^{p_0 -1} |v' (r)| dr=$$

$$= \frac{p_0 \omega_N}{N} \int_{0}^{R+1}  r^{\frac{N-1}{p}} |v'(r)| r^{\frac{a_0}{p}}r^{-\frac{a_0}{p}} r^{\frac{Np-N+1}{p}} |v(r)|^{p_0 -1} dr 
$$

$$
\leq \frac{p_0 \omega_N}{N} \left( \int_{0}^{R+1} r^{N-1} |v' (r)|^p r^{a_0} dr \right)^{1/p} \left( \int_{0}^{R+1} r^{\frac{Np-N+1-a_0}{p-1}} 
|v(r)|^{(p_0 -1)\frac{p}{p-1}} dr \right)^{\frac{p-1}{p}}
$$

$$
\leq  \frac{p_0 \omega_N}{N} \left( \int_{0}^{R+1} r^{N-1} A(r) |v' (r)|^p dr \right)^{1/p} \left( \int_{0}^{R+1} r^{N-1} |v(r)|^{p_0}
r^{\frac{p-a_0}{p-1}} |v(r)|^{\frac{p_0 -p}{p-1}} dr \right)^{\frac{p-1}{p}}
$$

$$
\leq C \left( \int_{B_{R+1}} A(|x|) |\nabla v (x)|^p dx \right)^{1/p} \times
$$
$$\times \left( \int_{0}^{R+1} r^{N-1} |v(r)|^{p_0} r^{\frac{p-a_0}{p-1}}  r^{-\frac{N+a_{0}-p}{p} \frac{p_{0}-p}{p-1} } \, \left(  \int_{B_{R+1}} A(|x|) \, |\nabla v (x)|^p dx \right)^{\frac{p_{0}-p}{p(p-1)}} \right)^{\frac{p-1}{p}} 
$$

$$
=  C \left( \int_{B_{R+1}} A(|x|) |\nabla v (x)|^p dx \right)^{\frac{p_0}{p^2}}  \left( \int_{B_{R+1}} |v(x)|^{p_0} dx \right)^{\frac{p-1}{p}} ,
$$

\noindent 
Notice that $\frac{p-a_0}{p-1} -\frac{N+a_{0}-p}{p} \frac{p_{0}-p}{p-1} =0$, by the definition of $p_0$. From these computations we easily deduce, as before, that

$$
\left( \int_{B_{R+1}}  |v (x)|^{p_0} dx \right)^{\frac{1}{p_0}}  \leq C \left( \int_{B_{R+1}} A(|x|) |\nabla v (x)|^p dx \right)^{\frac{1}{p}}. 
$$

\noindent Now, as in Lemma \ref{ALem}, we use the fact that $A$ is continuous and strictly positive on the compact set ${\overline {B_{R+1} \backslash B_{R} }}$, and we get

$$
\int_{B_{R+1}} A(|x|) |\nabla  \rho u|^p dx \leq C \int_{B_{R+1}} A(|x|) \rho^p |\nabla   u|^p dx + C \int_{B_{R+1} \backslash B_{R}} A(|x|) |u|^p |\nabla \rho |^p dx
\leq
$$

$$
\leq C ||u||^{p}_A + C \int_{B_{R}^c} A(|x|) | \nabla u|^p dx \left( \int_{B_{R+1} \backslash B_{R}} x^{-N-a_{\infty}+p}dx \right)^{1/p} \leq C ||u||^{p}_A .
$$

\noindent Hence we get

$$
\left( \int_{B_{R+1}} |\rho u|^{p_0} dx \right)^{1/p_0} \leq C ||u||_A 
$$

\noindent and therefore

$$
\left( \int_{B_{R}} | u|^{p_0} dx \right)^{1/p_0} = \left( \int_{B_{R}} |\rho u|^{p_0} dx \right)^{1/p_0} \leq \left( \int_{B_{R+1}} |\rho u|^{p_0} dx 
\right)^{1/p_0} \leq C ||u||_A ,
$$

\noindent which is the thesis.
\end{proof}

The following lemma gives another embedding result that we will use.

\begin{lem}
\label{A6}

Assume the hypothesis $\left( \mathbf{A}\right) $ and fix $0<r<R$. 
Then the embedding
$$
D_{A} \hookrightarrow L^{p}(B_R \setminus {\overline B_r })
$$
is continuous and compact.
\end{lem}

\begin{proof}
Set $E:=B_R \setminus {\overline B_r }$ for brevity.
The continuity of the embedding easily derives from (\ref{est1}), by integrating over the set $E$. 
As to compactness, let $\{u_{n}\}_n$ be a bounded sequence in $D_{A} $. By continuity of the embedding we obtain that also $\{|| u_n  ||_{L^{p}(E)} \}_n$ is bounded.
Moreover, as $A$ is continuous and strictly positive on the compact set $\overline {E }$, we have

$$
\int_{E} |\nabla u_n |^p dx \leq C \int_{ E } A(|x|) |\nabla u_n |^p dx
 \leq C_1 .
$$

\noindent Thus $\{u_{n}\}_n$ is bounded also in the space $W^{1,p}(E )$. Thanks to Rellich's Theorem, $\{u_{n}\}_n$ has a convergent subsequence in $L^{p}(E )$, and this gives the thesis.
\end{proof}


\section{Compactness results for the space $X$}\label{COMP}

In this section we state the main compactness results of this paper, concerning the space $X$. 
Recall that we define such a space as
$$X := D_A \cap  L^{p}(\mathbb{R}^{N}, V (|x|) dx ) $$
endowed with the norm $||u|| := \left( ||u||_{A}^{p} + ||u||_{L^{p}(\mathbb{R}^{N}, V (|x|) dx )}^{p} \right)^{1/p}$, with respect to which $X$ is a Banach space. 
The compactness results that we state here will be proved in sections \ref{SEC:1} and \ref{SEC:2}.

Given $A$, $V$ and $K$ as in $\left( \mathbf{A}\right) $, $\left( \mathbf{V}\right) $ and $\left( \mathbf{K}\right) $, we define the following functions of $R>0$ and $q>1$: 
\begin{eqnarray}
\mathcal{S}_{0}\left( q,R\right)&:=&
\sup_
{u\in X,\,
\left\| u\right\| =1  }
\int_{B_{R}}K\left( \left| x\right| \right)
\left| u\right| ^{q}dx,  \label{S_o :=}
\\
\mathcal{S}_{\infty }\left( q,R\right)&:=&
\sup_
{u\in X,\,
\left\| u\right\| =1  } 
\int_{\mathbb{R}%
^{N}\setminus B_{R}}K\left( \left| x\right| \right) \left| u\right| ^{q}dx.
\label{S_i :=}
\end{eqnarray}
Clearly $\mathcal{S}_{0}\left( q,\cdot \right) $ is nondecreasing, $\mathcal{
S}_{\infty }\left( q,\cdot \right) $ is nonincreasing and both of them can
be infinite at some $R$.

Our first result concerns the embedding properties of $X$ into the sum
space 
\[
L_{K}^{q_{1}}+L_{K}^{q_{2}}:=\left\{ u_{1}+u_{2}:u_{1}\in
L_{K}^{q_{1}}\left( \mathbb{R}^{N}\right) ,\,u_{2}\in L_{K}^{q_{2}}\left( \mathbb{R%
}^{N}\right) \right\} ,\quad 1<q_{i}<\infty . 
\]
We recall from \cite{BPR} that such a space can be characterized as the set
of measurable mappings $u:\mathbb{R}^{N}\rightarrow \mathbb{R}$ for which there
exists a measurable set $E\subseteq \mathbb{R}^{N}$ such that $u\in
L_{K}^{q_{1}}\left( E\right) \cap L_{K}^{q_{2}}\left( E^{c}\right) $. It is
a Banach space with respect to the norm 
\[
\left\| u\right\| _{L_{K}^{q_{1}}+L_{K}^{q_{2}}}:=\inf_{u_{1}+u_{2}=u}\max
\left\{ \left\| u_{1}\right\| _{L_{K}^{q_{1}}(\mathbb{R}^{N})},\left\|
u_{2}\right\| _{L_{K}^{q_{2}}(\mathbb{R}^{N})}\right\} 
\]
and the continuous embedding $L_{K}^{q}\hookrightarrow
L_{K}^{q_{1}}+L_{K}^{q_{2}}$ holds for all $q\in \left[ \min \left\{
q_{1},q_{2}\right\} ,\max \left\{ q_{1},q_{2}\right\} \right] $. The
assumptions of our result are quite general but not so easy to check, so more handy conditions ensuring these general assumptions
will be provided by the next results.


\begin{thm}
\label{THM(cpt)} Let $1<p<N$, let $A$, $V$ and $K$ be as in 
$\left( \mathbf{A}\right) $, $\left( \mathbf{V}\right) $ and $\left( \mathbf{K}\right) $, and let $q_{1},q_{2}>1$.

\begin{itemize}
\item[(i)]  If 
\begin{equation}
\mathcal{S}_{0}\left( q_{1},R_{1}\right) <\infty \quad \text{and}\quad 
\mathcal{S}_{\infty }\left( q_{2},R_{2}\right) <\infty \quad \text{for some }%
R_{1},R_{2}>0,  
\tag*{$\left( {\cal S}_{q_{1},q_{2}}^{\prime }\right) $}
\end{equation}
then $X$ is continuously embedded into $L_{K}^{q_{1}}(\mathbb{R}^{N})+L_{K}^{q_{2}}(\mathbb{R}^{N})$.

\item[(ii)]  If 
\begin{equation}
\lim_{R\rightarrow 0^{+}}\mathcal{S}_{0}\left( q_{1},R\right)
=\lim_{R\rightarrow +\infty }\mathcal{S}_{\infty }\left( q_{2},R\right) =0, 
\tag*{$\left({\cal S}_{q_{1},q_{2}}^{\prime \prime }\right) $}
\end{equation}
then $X$ is compactly embedded into $L_{K}^{q_{1}}(\mathbb{R}^{N})+L_{K}^{q_{2}}(\mathbb{R}^{N})$.

\end{itemize}
\end{thm}

\noindent Observe that, of course, $(\mathcal{S}_{q_{1},q_{2}}^{\prime \prime })$
implies $(\mathcal{S}_{q_{1},q_{2}}^{\prime })$. Moreover, these assumptions
can hold with $q_{1}=q_{2}=q$ and therefore Theorem \ref{THM(cpt)} also
concerns the embedding properties of $X$ into $L_{K}^{q}$, $1<q<\infty $.
\smallskip

We now look for explicit conditions on $V$ and $K$ implying $(\mathcal{S}%
_{q_{1},q_{2}}^{\prime \prime })$ for some $q_{1}$ and $q_{2}$. More
precisely, in Theorem \ref{THM0} we will find a range of exponents $%
q_{1} $ such that $\lim_{R\rightarrow 0^{+}}\mathcal{S}_{0}\left(q_{1},R\right)$ $=0$, while 
in Theorem \ref{THM1} we will do the same for exponents $q_{2}$ such that
$\lim_{R\rightarrow +\infty}\mathcal{S}_{\infty }\left( q_{2},R\right) =0$.


For $\alpha \in \mathbb{R}$, $\beta \in \left[ 0,1\right] $, $a>p-N$, we define two
functions $\alpha ^{*}\left(a, \beta \right) $ and $q^{*}\left(a, \alpha ,\beta
\right) $ by setting 
\[
\alpha ^{*}\left(a, \beta \right) :=\max \left\{ p\beta -1-\frac{p-1}{p}N -a \beta + \frac{a}{p},-\left(
1-\beta \right) N\right\} =\]
\[
\left\{ 
\begin{array}{ll}
p\beta -1-\frac{p-1}{p}N  -a\beta +\frac{a}{p}\quad \smallskip & \text{if }0\leq \beta \leq \frac{1}{p%
} \\ 
-\left( 1-\beta \right) N & \text{if }\frac{1}{p}\leq \beta \leq 1
\end{array}
\right. 
\]
and 
\[
q^{*}\left(a, \alpha ,\beta \right) :=p\frac{\alpha -p\beta +N+ a \beta}{N-p+a}. 
\]

\begin{thm}
\label{THM0}
Let $A$, $V$, $K$ be as in $\left( \mathbf{A}\right) $, $\left( \mathbf{V}\right) $, $\left( \mathbf{K}\right) $.
Assume that there exists $R_{1}>0$ such that $V\left( r\right) <+\infty $ almost everywhere in $(0,R_1)$ and
\begin{equation}
\esssup_{r\in \left( 0,R_{1}\right) }\frac{K\left( r\right) }{%
r^{\alpha _{0}}V\left( r\right) ^{\beta _{0}}}<+\infty \quad \text{for some }%
0\leq \beta _{0}\leq 1\text{~and }\alpha _{0}>\alpha ^{*}\left(a_{0}, \beta
_{0}\right) .  \label{esssup in 0}
\end{equation}
Then $\displaystyle \lim_{R\rightarrow 0^{+}}\mathcal{S}_{0}\left(
q_{1},R\right) =0$ for every $q_{1}\in \mathbb{R}$ such that 
\begin{equation}
\max \left\{ 1,p\beta _{0}\right\} <q_{1}<q^{*}\left(a_{0}, \alpha _{0},\beta
_{0}\right) .  \label{th1}
\end{equation}
\end{thm}

\begin{thm}
\label{THM1}
Let $A$, $V$, $K$ be as in $\left( \mathbf{A}\right) $, $\left( \mathbf{V}\right) $, $\left( \mathbf{K}\right) $.
Assume that there exists $R_{2}>0$ such that $V\left( r\right) <+\infty $ for almost every $r>R_2$ and
\begin{equation}
\esssup_{r>R_{2}}\frac{K\left( r\right) }{r^{\alpha _{\infty
}}V\left( r\right) ^{\beta _{\infty }}}<+\infty \quad \text{for some }0\leq
\beta _{\infty }\leq 1\text{~and }\alpha _{\infty }\in \mathbb{R}.
\label{esssup all'inf}
\end{equation}
Then $\displaystyle \lim_{R\rightarrow +\infty }\mathcal{S}_{\infty }\left(
q_{2},R\right) =0$ for every $q_{2}\in \mathbb{R}$ such that 
\begin{equation}
q_{2}>\max \left\{ 1,p\beta _{\infty },q^{*}\left(a_{\infty}, \alpha _{\infty },\beta
_{\infty }\right) \right\} .  \label{th2}
\end{equation}
\end{thm}

We observe explicitly that for every $a$, $\alpha$, $\beta $ as above one has 
\[
\max \left\{ 1,p\beta ,q^{*}\left(a, \alpha ,\beta \right) \right\} =\left\{ 
\begin{array}{ll}
q^{*}\left(a, \alpha ,\beta \right) \quad & \text{if }\alpha \geq \alpha
^{*}\left(a, \beta \right) \smallskip \\ 
\max \left\{ 1,p\beta \right\} & \text{if }\alpha \leq \alpha ^{*}\left(a,
\beta \right)
\end{array}
\right. . 
\]

\begin{rem}
\label{RMK: suff12}\quad 

\begin{enumerate}
\item  \label{RMK: suff12-V^0}We mean $V\left( r\right) ^{0}=1$ for every $r$
(even if $V\left( r\right) =0$). In particular, if $V\left( r\right) =0$ for
almost every $r>R_{2}$, then Theorem \ref{THM1} can be applied with $\beta
_{\infty }=0$ and assumption (\ref{esssup all'inf}) means 
\[
\esssup_{r>R_{2}}\frac{K\left( r\right) }{r^{\alpha _{\infty }}}%
<+\infty \quad \text{for some }\alpha _{\infty }\in \mathbb{R}.
\]
Similarly for Theorem \ref{THM0} and assumption (\ref{esssup in 0}), if $%
V\left( r\right) =0$ for almost every $r\in \left( 0,R_{1}\right) $.

\item  \label{RMK: suff12-no hp}The inequality $\max \left\{ 1,p\beta
_{0}\right\} <q^{*}\left(a_0 , \alpha _{0},\beta _{0}\right) $ is equivalent to $%
\alpha _{0}>\alpha ^{*}\left(a_0 , \beta _{0}\right) $. Then, in (\ref{th1}),
such inequality is automatically true and does not ask for further
conditions on $\alpha _{0}$ and $\beta _{0}$.

\item  \label{RMK: suff12-Vbdd}The assumptions of Theorems \ref{THM0} and 
\ref{THM1} may hold for different pairs $\left( \alpha _{0},\beta_{0}\right)$,
$\left( \alpha _{\infty },\beta _{\infty }\right) $ (assuming $p$ and $a_0$ fixed). In this
case, of course, one chooses them in order to get the ranges for $q_{1},q_{2}
$ as large as possible. For instance, assume that $a_0 \leq p$ and $V$ is not singular at the origin, 
i.e., $V$ is essentially bounded in a neighbourhood of 0. 
If condition (\ref{esssup in 0}) holds true for a pair $\left( \alpha _{0},\beta _{0}\right) $%
, then (\ref{esssup in 0}) also holds for all pairs $\left( \alpha
_{0}^{\prime },\beta _{0}^{\prime }\right) $ such that $\alpha _{0}^{\prime
}>\alpha _{0}$ and $\beta _{0}^{\prime }<\beta _{0}$. Therefore, since $\max
\left\{ 1,p\beta \right\} $ is nondecreasing in $\beta $ and $q^{*}\left(a, 
\alpha ,\beta \right) $ is increasing in $\alpha $ and decreasing in $\beta $ (because $a_0 \leq p$)
, it is convenient to choose $\beta _{0}=0$ and the best interval where one
can take $q_{1}$ is $1<q_{1}<q^{*}\left(a_0 , \overline{\alpha },0\right) $ with $%
\overline{\alpha }:=\sup \left\{ \alpha _{0}:\esssup_{r\in \left(
0,R_{1}\right) }K\left( r\right) /r^{\alpha _{0}}<+\infty \right\} $
(here we mean $q^{*}\left(a_0 , +\infty ,0\right) =+\infty $).
\end{enumerate}
\end{rem}


\section{Proof of Theorem \ref{THM(cpt)} \label{SEC:1}}

Assume as usual $N\geq 3$ and $1<p<N$, and let $A$, $V$ and $K$ be as in $\left( \mathbf{A}\right) $, $\left( \mathbf{V}\right) $
and $\left( \mathbf{K}\right) $. 
Recall from assumption $\left( \mathbf{K}\right) $ that $K\in L_{\mathrm{loc}}^{s}\left( \left( 0,+\infty \right) \right) $ for some $s>1$.

\begin{lem}
\label{Lem(corone)}Let $R>r>0$ and $1<q<\infty $. 
Then there exist $\tilde{C}=\tilde{C}\left(N,p,r,R,q,s\right) >0$ and $l=l\left(p,q,s\right) >0$
such that $q-lp>0$ and $\forall u\in X$ one has 
\begin{equation}
\int_{B_{R}\setminus B_{r}}K\left( \left| x\right| \right) \left|u\right| ^q dx
\leq
\tilde{C}\left\| K\left( \left| \cdot\right| \right) \right\| _{L^{s}(B_{R}\setminus B_{r})}
\left\|u\right\|^{q-lp} \left(\int_{B_R\setminus B_r}\left|u\right|^{p}dx\right)^l.
\label{LEM:corone}
\end{equation}
\end{lem}

Notice that, in the second part of Lemma \ref{LEM:corone}, 
$s>\frac{Np}{N(p-1)+p-a_+}$ implies $\tilde{q}>1$.

\proof%
Let $u\in X$ and fix $t\in(1,s)$ such that $t'q>p$ (where $t'=t/(t-1)$). 
Then, by H\"{o}lder inequality and the pointwise estimates of Section 2, we have 
\begin{eqnarray*}
&& \int_{B_{R}\setminus B_{r}}K\left( \left| x\right| \right) \left| u\right|^{q} dx \\
&\leq & \left( \int_{B_{R}\setminus B_{r}}K\left(\left| x\right| \right) ^{t}dx\right) ^{\frac{1}{t}}
\left( \int_{B_{R}\setminus B_{r}}\left|u\right| ^{t^{\prime }q}dx\right) ^{\frac{1}{t^{\prime }}} \\
&\leq & 
\left| B_{R}\setminus B_{r}\right| ^{\frac{1}{t}-\frac{1}{s}}
\left\| K\left( \left| \cdot \right| \right) \right\|_{L^{s}(B_{R}\setminus B_{r})}
\left( \int_{B_{R}\setminus B_{r}}\left|u\right| ^{t^{\prime }q-p}\left|u\right|^{p} dx\right) ^{\frac{1}{t^{\prime }}}
\\
&\leq &
\left| B_{R}\setminus B_{r}\right| ^{\frac{1}{t}-\frac{1}{s}}
\left\| K\left( \left| \cdot \right| \right) \right\|_{L^{s}(B_{R}\setminus B_{r})}
\left( \frac{C \left\|u\right\| }{r^{\frac{N-p+a_0}{p}}}\right) ^{q-p/t^{\prime}}
\left( \int_{B_{R}\setminus B_{r}}\left|u\right|^{p} dx\right) ^{\frac{1}{t^{\prime }}}.
\end{eqnarray*}
This proves (\ref{LEM:corone}). 
\endproof

We now prove Theorem \ref{THM(cpt)}. Recall the definitions (\ref{S_o :=})-(\ref{S_i :=}) of the functions $\mathcal{S}_{0}$ and 
$\mathcal{S}_{\infty }$, and the following result from \cite{BPR} concerning convergence in the sum of Lebesgue spaces.

\begin{prop}[{\cite[Proposition 2.7]{BPR}}] 
\label{Prop(->0)}
Let $\left\{ u_{n}\right\}
\subseteq L_{K}^{p_{1}}+L_{K}^{p_{2}}$ be a sequence such that $\forall
\varepsilon >0$ there exist $n_{\varepsilon }>0$ and a sequence of
measurable sets $E_{\varepsilon ,n}\subseteq \mathbb{R}^{N}$ satisfying 
\begin{equation}
\forall n>n_{\varepsilon },\quad \int_{E_{\varepsilon ,n}}K\left( \left|
x\right| \right) \left| u_{n}\right| ^{p_{1}}dx+\int_{E_{\varepsilon
,n}^{c}}K\left( \left| x\right| \right) \left| u_{n}\right|
^{p_{2}}dx<\varepsilon .  \label{Prop(->0): cond}
\end{equation}
Then $u_{n}\rightarrow 0$ in $L_{K}^{p_{1}}+L_{K}^{p_{2}}$.
\end{prop}

\proof[Proof of Theorem \ref{THM(cpt)}]
We prove each part of the theorem separately.\smallskip

\noindent (i) By the monotonicity of $\mathcal{S}_{0}$ and $\mathcal{S}%
_{\infty }$, it is not restrictive to assume $R_{1}<R_{2}$ in hypothesis $%
\left( \mathcal{S}_{q_{1},q_{2}}^{\prime }\right) $. In order to prove the
continuous embedding, let $u\in X$, $u\neq 0$. Then we
have 
\begin{equation}
\int_{B_{R_{1}}}K\left( \left| x\right| \right) \left| u\right|
^{q_{1}}dx=\left\| u\right\| ^{q_{1}}\int_{B_{R_{1}}}K\left( \left| x\right|
\right) \frac{\left| u\right| ^{q_{1}}}{\left\| u\right\| ^{q_{1}}}dx\leq
\left\| u\right\| ^{q_{1}}\mathcal{S}_{0}\left( q_{1},R_{1}\right) 
\label{pf1}
\end{equation}
and, similarly, 
\begin{equation}
\int_{B_{R_{2}}^{c}}K\left( \left| x\right| \right) \left| u\right|
^{q_{2}}dx\leq \left\| u\right\| ^{q_{2}}\mathcal{S}_{\infty }\left(
q_{2},R_{2}\right) .  \label{pf2}
\end{equation}
We now use (\ref{LEM:corone}) of Lemma \ref{Lem(corone)} and Lemma \ref{A6} to deduce that there exists a constant $C_{1}>0$, 
independent from $u$, such that 
\begin{equation}
\int_{B_{R_{2}}\setminus B_{R_{1}}}K\left( \left| x\right| \right) \left|
u\right| ^{q_{1}}dx\leq C_{1}\left\| u\right\| ^{q_{1}}.  \label{pf3}
\end{equation}
Hence $u\in L_{K}^{q_{1}}(B_{R_{2}})\cap L_{K}^{q_{2}}(B_{R_{2}}^{c})$ and
thus $u\in L_{K}^{q_{1}}+L_{K}^{q_{2}}$. Moreover, if $u_{n}\rightarrow 0$
in $X$, then, using (\ref{pf1}), (\ref{pf2}) and (\ref
{pf3}), we get 
\[
\int_{B_{R_{2}}}K\left( \left| x\right| \right) \left| u_{n}\right|
^{q_{1}}dx+\int_{B_{R_{2}}^{c}}K\left( \left| x\right| \right) \left|
u_{n}\right| ^{q_{2}}dx=o\left( 1\right) _{n\rightarrow \infty },
\]
which means $u_{n}\rightarrow 0$ in $L_{K}^{q_{1}}+L_{K}^{q_{2}}$ by
Proposition \ref{Prop(->0)}. \emph{\smallskip }

\noindent (ii) Assume hypothesis $\left( \mathcal{S}_{q_{1},q_{2}}^{\prime
\prime }\right) $. Let $\varepsilon >0$ and let $u_{n}\rightharpoonup 0$ in $
X$. Then $\left\{ \left\| u_{n}\right\| \right\}_n $ is
bounded and, arguing as for (\ref{pf1}) and (\ref{pf2}), we can take $
r_{\varepsilon }>0$ and $R_{\varepsilon }>r_{\varepsilon }$ such that for
all $n$ one has 
\[
\int_{B_{r_{\varepsilon }}}K\left( \left| x\right| \right) \left|
u_{n}\right| ^{q_{1}}dx\leq \left( \left\| u_{n}\right\| ^{q_{1}}\right) \mathcal{S}%
_{0}\left( q_{1},r_{\varepsilon }\right) \leq \left( \sup_{n}\left\| u_{n}\right\|
^{q_{1}}\right) \mathcal{S}_{0}\left( q_{1},r_{\varepsilon }\right) <\frac{%
\varepsilon }{3}
\]
and 
\[
\int_{B_{R_{\varepsilon }}^{c}}K\left( \left| x\right| \right) \left|
u_{n}\right| ^{q_{2}}dx\leq \left( \sup_{n}\left\| u_{n}\right\| ^{q_{2}} \right) \mathcal{S}
_{\infty }\left( q_{2},R_{\varepsilon }\right) <\frac{\varepsilon }{3}.
\]
Using (\ref{LEM:corone}) of Lemma \ref{Lem(corone)} and the boundedness of $\left\{ \left\|
u_{n}\right\| \right\} $ again, we infer that there exist two constants $
C_{2},l>0$, independent from $n$, such that 
\[
\int_{B_{R_{\varepsilon }}\setminus B_{r_{\varepsilon }}}K\left( \left|
x\right| \right) \left| u_{n}\right| ^{q_{1}}dx\leq C_{2}\left(
\int_{B_{R_{\varepsilon }}\setminus B_{r_{\varepsilon }}}\left| u_{n}\right|
^{p}dx\right) ^{l},
\]
where 
\[
\int_{B_{R_{\varepsilon }}\setminus B_{r_{\varepsilon }}}\left| u_{n}\right|
^{p}dx\rightarrow 0\quad \text{as }n\rightarrow \infty \quad \text{(}%
\varepsilon ~\text{fixed)}
\]
thanks to Lemma \ref{A6}. Therefore we obtain 
\[
\int_{B_{R_{\varepsilon }}}K\left( \left| x\right| \right) \left|
u_{n}\right| ^{q_{1}}dx+\int_{B_{R_{\varepsilon }}^{c}}K\left( \left|
x\right| \right) \left| u_{n}\right| ^{q_{2}}dx<\varepsilon 
\]
for all $n$ sufficiently large, which means $u_{n}\rightarrow 0$ in $%
L_{K}^{q_{1}}+L_{K}^{q_{2}}$ (Proposition \ref{Prop(->0)}). This concludes
the proof of part (ii).
\endproof

\section{Proof of Theorems \ref{THM0} and \ref{THM1} \label{SEC:2}}

Assume as usual $N \geq 3$ and $1<p<N$, and let $A$, $V$ and $K$ be as in $\left( \mathbf{A}\right)$, $\left( \mathbf{V}\right)$ and $\left( \mathbf{K}\right)$. 
\begin{lem}
\label{Lem(Omega)}Let $R_0>0$ and assume that $V(r)<+\infty$ almost everywhere in $B_{R_0}$ and
\[
\Lambda :=\esssup_{x\in B_{R_0} }\frac{K\left( \left| x\right|
\right) }{\left| x\right| ^{\alpha }V\left( \left| x\right| \right) ^{\beta }%
}<+\infty \quad \text{for some }0\leq \beta \leq 1\text{~and }\alpha \in 
\mathbb{R}.
\]
Let $u\in X$ and assume that there exist $\nu \in \mathbb{R}$ and $m>0$
such that 
\[
\left| u\left( x\right) \right| \leq \frac{m}{\left| x\right| ^{\nu }}\quad 
\text{almost everywhere in } B_{R_0} .
\]
Then there exists a constant $C=C(N, R_0, a_0, a_{\infty} ,\beta )>0$ such that $\forall R \in (0,R_{0})$ and $\forall q>\max \left\{ 1,p\beta \right\} $, one has 
\bigskip

$\displaystyle\int_{B_R }K\left( \left| x\right| \right) \left| u\right| ^{q} dx$
\[
\leq \left\{ 
\begin{array}{ll}
\Lambda m^{q-1} C \left( \int_{B_R}\left| x\right| ^{\frac{%
\alpha -\nu \left( q-1\right) }{N(p-1)+p\left( 1-p\beta +a_0 \beta \right) -a_0}pN}dx\right)^{
\frac{N(p-1)+p\left( 1-p\beta +a_0 \beta \right) -a_0}{pN}}\left\| u\right\| \quad \medskip  & 
\text{if }0\leq \beta \leq \frac{1}{p} \\ 
\Lambda m^{q-p\beta }\left( \int_{B_R }\left| x\right| ^{\frac{\alpha
-\nu \left( q-p\beta \right) }{1-\beta }}dx\right) ^{1-\beta }\left\|
u\right\| ^{p\beta } \medskip  & \text{if }\frac{1}{p}
<\beta <1 \\ 
\Lambda m^{q-p}\left( \int_{B_R }\left| x\right| ^{\frac{p}{p-1}(\alpha -\nu \left(
q-p\right)) }V\left( \left| x\right| \right) \left| u\right| ^{p}dx\right) ^{
\frac{p-1}{p}}\left\| u\right\|  & \text{if }\beta =1.
\end{array}
\right. 
\]
\end{lem}

\begin{proof}

We distinguish several cases, where we will use H\"{o}lder inequality many
times, without explicitly noting it. \smallskip

\noindent \emph{Case }$\beta =0$\emph{. }

\noindent We apply H\"older inequality with exponents $p_0 = \frac{Np}{N-p +a_0}$, $(p_0 )' = \frac{Np}{N(p-1) -a_0 +p}$, and we use Lemma \ref{A4}. We have 
{\allowdisplaybreaks
\begin{eqnarray*}
\frac{1}{\Lambda }\int_{B_R }K\left( \left| x\right| \right) \left|
u\right| ^{q} dx 
&\leq & \int_{B_R }\left| x\right|^{\alpha }\left| u\right| ^{q-1}\left| u\right| dx \\
&\leq & \left( \int_{B_R
}\left( \left| x\right| ^{\alpha }\left| u\right| ^{q-1}\right) ^{\frac{pN}{
N(p-1)+p-a_0}}dx\right) ^{\frac{N(p-1)+p-a_0}{pN}}\left( \int_{B_R }\left| u\right|
^{p_0}dx\right) ^{\frac{1}{p_0}} \\
&\leq & m^{q-1} C \left( \int_{B_R }\left| x\right| ^{\frac{\alpha -\nu
\left( q-1\right) }{N(p-1)+p- a_0}pN}dx\right) ^{\frac{N(p-1)+p -a_0}{pN}}\left\| u\right\| .
\end{eqnarray*}
}

\smallskip

\noindent \emph{Case }$0<\beta <1/p$\emph{.}%
\smallskip

\noindent One has $\frac{1}{\beta }>1$ and $\frac{1-\beta }{1-p\beta }
p_0 >1 $, with H\"{o}lder conjugate exponents $\left( \frac{1}{\beta }
\right) ^{\prime }=\frac{1}{1-\beta }$ and $\left( \frac{1-\beta }{1-p\beta }
p_0\right) ^{\prime }=\frac{pN\left( 1-\beta \right) }{N(p-1)+p\left( 1-p\beta
+a_0 \beta\right) -a_0}$. Then we get 
{\allowdisplaybreaks
\begin{eqnarray*}
&&\frac{1}{\Lambda }\int_{B_R }K\left( \left| x\right| \right) \left|
u\right| ^{q} dx 
\\
&\leq &\int_{B_R }\left| x\right| ^{\alpha }V\left( \left| x\right|
\right) ^{\beta }\left| u\right| ^{q} dx=\int_{B_R
}\left| x\right| ^{\alpha }\left| u\right| ^{q-1}\left| u\right| ^{1-p\beta
}V\left( \left| x\right| \right) ^{\beta }\left| u\right| ^{p\beta }dx 
\\
&\leq &\left( \int_{B_R }\left( \left| x\right| ^{\alpha }\left| u\right|
^{q-1}\left| u\right| ^{1-p\beta }\right) ^{\frac{1}{1-\beta }}dx\right)
^{1-\beta }\left( \int_{B_R }V\left( \left| x\right| \right) \left|
u\right| ^{p}dx\right) ^{\beta } 
\\
&\leq & \left( \left( \int_{B_R}\left( \left| x\right| ^{\frac{\alpha }{
1-\beta }}\left| u\right| ^{\frac{q-1}{1-\beta }}\right) ^{\left( \frac{
1-\beta }{1-p\beta }p_0\right) ^{\prime }}dx\right) ^{\frac{1}{\left( 
\frac{1-\beta }{1-p\beta }p_0\right) ^{\prime }}}\left( \int_{B_R
}\left| u\right| ^{p_0}dx\right) ^{\frac{1-p\beta }{\left( 1-\beta \right)
p_{0}}}\right) ^{1-\beta }\left\| u\right\| ^{p\beta } 
\\
&\leq &m^{q-1} C \left( \left( \int_{B_R}\left( \left| x\right| ^{\frac{
\alpha }{1-\beta }-\nu \frac{q-1}{1-\beta }}\right) ^{\left( \frac{1-\beta }{
1-p\beta }p_{0}\right) ^{\prime }}dx\right) ^{\frac{1}{\left( \frac{1-\beta 
}{1-\beta }p_{0}\right) ^{\prime }}} C
\left\| u\right\| ^{\frac{1-p\beta }{1-\beta }}\right)^{1-\beta }\left\|
u\right\| ^{p\beta } 
\\
&=&  m^{q-1} C \left( \int_{B_R }\left| x\right| ^{\frac{\alpha -\nu \left(
q-1\right) }{N(p-1)+p\left( 1-p\beta +a_0 \beta \right)-a_0 }pN}dx\right) ^{\frac{N(p-1)+p\left(
1-p\beta +a_0 \beta \right) -a_0}{pN}} C \left\| u\right\| .
\end{eqnarray*}
}
\smallskip

\noindent \emph{Case }$\beta =\frac{1}{p}$\emph{.}
\smallskip

\noindent We have 
{\allowdisplaybreaks
\begin{eqnarray*}
\frac{1}{\Lambda }\int_{B_R }K\left( \left| x\right| \right) \left|
u\right| ^{q} dx &\leq &\int_{B_R }\left| x\right|
^{\alpha }\left| u\right| ^{q-1}V\left( \left| x\right| \right) ^{\frac{1}{p}%
}\left| \right| dx \\u
&\leq &\left( \int_{B_R }\left| x\right| ^{\alpha \frac{p}{p-1}}\left| u\right|
^{\left( q-1\right)\frac{p}{p-1} }dx\right) ^{\frac{p-1}{p}}\left( \int_{B_R }V\left(
\left| x\right| \right) \left| u\right| ^{p}dx\right) ^{\frac{1}{p}} \\
&\leq &m^{q-1}\left( \int_{B_R}\left| x\right| ^{(\alpha -\nu \left(
q-1\right) )\frac{p}{p-1} }dx\right) ^{\frac{p-1}{p}}\left\| u\right\| .
\end{eqnarray*}
}

\noindent \emph{Case }$1/p<\beta <1$.

\noindent One has $\frac{p-1}{p\beta -1}>1$, with H\"{o}lder conjugate
exponent $\left( \frac{p-1}{p\beta -1}\right) ^{\prime }=\frac{p-1}{p\left(
1-\beta \right) }$. Then 
{\allowdisplaybreaks
\begin{eqnarray*}
&&\frac{1}{\Lambda }\int_{B_R }K\left( \left| x\right| \right) \left|
u\right| ^{q}dx 
\\
&\leq &\int_{B_R }\left| x\right|
^{\alpha }V\left( \left| x\right| \right) ^{\beta }\left| u\right|
^{q} dx=\int_{B_R }\left| x\right| ^{\alpha }V\left(
\left| x\right| \right) ^{\frac{p\beta -1}{p}}\left| u\right| ^{q-1}V\left(
\left| x\right| \right) ^{\frac{1}{p}}\left| u\right| dx 
\\
&\leq &\left( \int_{B_R}\left| x\right| ^{\alpha \frac{p}{p-1}}V\left( \left|
x\right| \right) ^{\frac{p\beta -1}{p-1}}\left| u\right| ^{\left( q-1\right) \frac{p}{p-1}}dx\right)
^{\frac{p-1}{p}}\left( \int_{B_R }V\left( \left| x\right| \right) \left|
u\right| ^{p}dx\right) ^{\frac{1}{p}} 
\\
&\leq &\left( \int_{B_R }\left| x\right| ^{\alpha \frac{p}{p-1}}\left| u\right|
^{(q-1)\frac{p}{p-1}- p\frac{p\beta -1}{p-1} }V\left( \left| x\right| \right) ^{\frac{p\beta -1}{p-1}}\left| u\right| ^{p \frac{p\beta -1}{p-1} }dx\right) ^{\frac{p-1}{p}}\left\| u\right\| 
\\
&\leq &\left( \left( \int_{B_R }\left| x\right| ^{\frac{\alpha }{1-\beta }%
}\left| u\right| ^{\frac{q-p\beta }{1-\beta }}dx\right) ^{\frac{p}{p-1}\left( 1-\beta
\right) }\left( \int_{B_R  }V\left( \left| x\right| \right) \left|
u\right| ^{p}dx\right) ^{\frac{p\beta -1}{p-1}}\right) ^{\frac{p-1}{p}}\left\| u\right\| 
\\
&\leq &m^{q-p\beta }\left( \int_{B_R  }\left| x\right|^{\frac{\alpha }{1-\beta }-\nu \frac{q-p\beta }{1-\beta }}
dx\right) ^{1-\beta }\left( \int_{B_R }V\left(
\left| x\right| \right) \left| u\right| ^{p}dx\right) ^{\frac{p\beta -1}{p}%
}\left\| u\right\| 
\\
&\leq &m^{q-p\beta }\left( \int_{B_R }\left| x\right| ^{\frac{\alpha -\nu
(q-p\beta )}{1-\beta }}dx\right) ^{1-\beta }\left\| u\right\| ^{p\beta
-1}\left\| u\right\| .
\end{eqnarray*}
}

\noindent \emph{Case }$\beta =1$\emph{.}

\noindent Assumption $q>\max \left\{ 1,p\beta \right\} $ means $q>p$ and
thus we have 
{\allowdisplaybreaks
\begin{eqnarray*}
&&\frac{1}{\Lambda }\int_{B_R  }K\left( \left| x\right| \right) \left|
u\right| ^{q} dx 
\\
&\leq &\int_{B_R  }\left| x\right|
^{\alpha }V\left( \left| x\right| \right) \left| u\right| ^{q} dx 
= \int_{B_R  }\left| x\right| ^{\alpha }V\left( \left| x\right|
\right) ^{\frac{p-1}{p}}\left| u\right| ^{q-1}V\left( \left| x\right| \right)
^{\frac{1}{p}}\left| u\right| dx 
\\
&\leq &\left( \int_{B_R  }\left| x\right| ^{\alpha \frac{p}{p-1}}V\left( \left|
x\right| \right) \left| u\right| ^{\left( q-1\right)\frac{p}{p-1} }dx\right) ^{\frac{p-1}{p%
}}\left( \int_{B_R  }V\left( \left| x\right| \right) \left| u\right|
^{p}dx\right) ^{\frac{1}{p}} 
\\
&\leq &\left( \int_{B_R  }\left| x\right| ^{\alpha \frac{p}{p-1}}\left| u\right|
^{(q-1)\frac{p}{p-1}-p }V\left( \left| x\right| \right) \left| u\right|
^{p}dx\right) ^{\frac{p-1}{p}}\left\| u\right\| 
\\
&\leq &m^{q-p}\left( \int_{B_R }\left| x\right| ^{\frac{p}{p-1}\left(\alpha -\nu (q-p) \right) }V\left( \left| x\right| \right) \left| u\right| ^{p}dx\right) ^{\frac{p-1}{p}}\left\| u\right\| .
\end{eqnarray*}
}
\end{proof}

\noindent The following lemma is analogous to the previous one, so we skip its proof for brevity.

\begin{lem}
\label{Lem(Omega2)}Let $R_0>0$ and assume that $V(r)<+\infty$ almost everywhere in $B_{R_0^c}$ and
\[
\Lambda :=\esssup_{x\in B_{R_0^c} }\frac{K\left( \left| x\right|
\right) }{\left| x\right| ^{\alpha }V\left( \left| x\right| \right) ^{\beta }%
}<+\infty \quad \text{for some }0\leq \beta \leq 1\text{~and }\alpha \in 
\mathbb{R}.
\]
Let $u\in X$ and assume that there exist $\nu \in \mathbb{R}$ and $m>0$
such that 
\[
\left| u\left( x\right) \right| \leq \frac{m}{\left| x\right| ^{\nu }}\quad 
\text{almost everywhere on } B_{R_0^c} .
\]
Then there exists a constant $C=C(N, R_0, a_0, a_{\infty} ,\beta )>0$ such that $\forall R > R_{0}$ and $\forall q>\max \left\{ 1,p\beta \right\} $, one has 

$\displaystyle\int_{B_R^c }K\left( \left| x\right| \right) \left| u\right| ^{q-1}\left|
h\right| dx$
\[
\leq \left\{ 
\begin{array}{ll}
\Lambda m^{q-1} C \left( \int_{B^c_R}\left| x\right| ^{\frac{%
\alpha -\nu \left( q-1\right) }{N(p-1)+p\left( 1-p\beta +a_{\infty} \beta \right) -a_{\infty}}pN}dx\right)^{
\frac{N(p-1)+p\left( 1-p\beta +a_{\infty} \beta \right)-a_{\infty} }{pN}}\left\| u\right\| \quad \medskip  & 
\text{if }0\leq \beta \leq \frac{1}{p} \\ 
\Lambda m^{q-p\beta }\left( \int_{B^c_R }\left| x\right| ^{\frac{\alpha
-\nu \left( q-p\beta \right) }{1-\beta }}dx\right) ^{1-\beta }\left\|
u\right\| ^{p\beta } \medskip  & \text{if }\frac{1}{p}
<\beta <1 \\ 
\Lambda m^{q-p}\left( \int_{B^c_R }\left| x\right| ^{\frac{p}{p-1}(\alpha -\nu \left(
q-p\right)) }V\left( \left| x\right| \right) \left| u\right| ^{p}dx\right) ^{
\frac{p-1}{p}}\left\| u\right\|  & \text{if }\beta =1.
\end{array}
\right. 
\]
\end{lem}

\medskip

We can now prove Theorems \ref{THM0} and \ref{THM1}.

\begin{proof}[Proof of Theorem \ref{THM0}]

Assume the hypotheses of the theorem and let $u\in X$ be
such that $\left\| u\right\| =1$. Let $0<R< R_{1}$. We
will denote by $C$ any positive constant which does not depend on $u$ and $R$.

\noindent Recalling the pointwise estimates of Lemma \ref{lemdens2} and the fact that 
\[
\esssup_{x\in B_{R}}\frac{K\left( \left| x\right| \right) }{\left|
x\right| ^{\alpha _{0}}V\left( \left| x\right| \right) ^{\beta _{0}}}\leq 
\esssup_{r\in \left( 0,R_{1}\right) }\frac{K\left( r\right) }{%
r^{\alpha _{0}}V\left( r\right) ^{\beta _{0}}}<+\infty , 
\]
we can apply Lemma \ref{Lem(Omega)} with $R_0=R_1$, $\alpha =\alpha
_{0}$, $\beta =\beta _{0}$, $m=M\left\| u\right\| =M$ and $\nu =
\frac{N-p+a_0}{p}$. 
\par \noindent If $0\leq \beta _{0}\leq 1/p$ we get

$$
\int_{B_{R}}K\left( \left| x\right| \right) \left| u\right| ^{q_{1}} dx 
\leq C\left( \int_{B_{R}}\left| x\right|^{\frac{\alpha _{0}-
\nu\left( q_{1}-1\right) }{N(p-1)-a_0 +p\left( 1-p\beta _{0} +a_0 \beta_0 \right) }
pN}dx\right)^{\frac{N(p-1)-a_0 +p\left( 1-p\beta _{0} +a_0 \beta_0 \right) }{pN}} $$

$$\leq  C\left( \int_{0}^R r^{\frac{\alpha _{0}-
\nu\left( q_{1}-1\right) }{N(p-1)-a_0 +p\left( 1-p\beta _{0} +a_0 \beta_0 \right) }
pN  +N -1} dr\right)^{\frac{N(p-1)-a_0 +p\left( 1-p\beta _{0} +a_0 \beta_0 \right) }{pN}} .
$$

\noindent Notice now that 

$$
{\frac{\alpha _{0}-
\nu\left( q_{1}-1\right) }{N(p-1)-a_0 +p\left( 1-p\beta _{0} +a_0 \beta_0 \right) }
pN  +N }= 
$$

$$= 
\frac{N}{N(p-1) -a_0 +p (1-p\beta_0 +a_0 \beta_0)} \left[  p(N+ \alpha_0 -p \beta_0 + \alpha_0 \beta_0 )- (N+a_0 -p) q_1     \right] =  
$$  

$$
=\frac{N (N+a_0 -p)}{N(p-1) -a_0 +p (1-p\beta_0 +a_0 \beta_0)} \left[ q^* (a_0 , \alpha_0 , \beta_0 )-q_1 \right] >0,
$$

\noindent thanks to the hypotheses. Hence we deduce

$$
\int_{B_{R}}K\left( \left| x\right| \right) \left| u\right| ^{q_{1}} dx 
\leq C R^{\frac{N+a_0 -p}{p}\left[ q^* (a_0 , \alpha_0 , \beta_0 )-q_1 \right]}.
$$

\noindent On the other hand, if $1/p<\beta _{0}<1$ we have 
\begin{eqnarray*}
\int_{B_{R}}K\left( \left| x\right| \right) \left| u\right| ^{q_{1}} dx 
&\leq & C\left( \int_{B_{R}}\left| x\right| ^{\frac{\alpha _{0}-
\nu\left( q_{1}-p\beta _{0}\right) }{1-\beta _{0}}}dx\right)^{1-\beta _{0}}\\
&\leq & C\left( \int_{0}^{R}r^{\frac{\alpha _{0}-\nu
\left( q_{1}-p\beta _{0}\right) }{1-\beta _{0}}+N-1}dr\right) ^{1-\beta _{0}},
\end{eqnarray*}
where

$$\frac{\alpha_0 -\nu (q_1 - p \beta_0 )}{1-\beta_0} +N = \frac{p\alpha_0 -(N+\alpha_0 -p) q_1 -p\beta_0}{p(1-\beta_0 )} +N=
$$

$$\frac{p\alpha_0 -(N+ a_0 -p) q_1 + Np\beta_0 + pa_0 \beta_0 -p^2 \beta_0 +Np -Np\beta_0 }{p(1-\beta_0 )}=$$

$$ \frac{p(\alpha_0 -p \beta_0 +N + a_0 \beta_0 )- (N+a_0 - p)q_1}{p(1-\beta_0 )}=
\frac{N+a_0 - p }{p(1-\beta_0 )} \left[ q^* (a_0 , \alpha_0 , \beta_0 )-q_1 \right] >0 .$$

\noindent Hence we get 

$$\int_{B_{R}}K\left( \left| x\right| \right) \left| u\right| ^{q_{1}} dx 
\leq C R^{\frac{N + a_0 -p}{p} \left[ q^* (a_0 , \alpha_0 , \beta_0 )-q_1 \right]}.$$

\noindent Finally, if $\beta _{0}=1$, we obtain 

$$
\int_{B_{R}} K\left( \left| x\right| \right) \left| u\right|^{q_{1}} dx \leq 
C \left( \int_{B_R}\left| x\right|^{ \frac{p}{p-1}(\alpha_0 -\nu ( q_1 -p ))  } V \left( \left| x\right| \right) |u|^{p} dx\right)^{\frac{p-1}{p}}, $$

\noindent where 

$$\alpha_0 - \nu (q_1 -p) = \alpha_0 -\frac{N+a_0 -p }{p}(q_1 -p)= \frac{1}{p} \left( p \alpha_0 -\left( N+a_0 -p) \right) q_1 +pN +p a_0 -p^2 \right)=$$

$$ \frac{1}{p} \left( p \left( \alpha_0 -p +N +a_0 \right) - \left( N+a_0 -p \right) q_1 \right)= \frac{N+ a_0 -p}{p} \left( q^* (a_0, \alpha_0 ,1 ) - q_1 \right) >0.$$

\noindent Hence we get

$$
\int_{B_{R}} K\left( \left| x\right| \right) \left| u\right|^{q_{1}}  dx \leq $$

$$\leq C\left( R^{\frac{N+a_0 -p}{p-1}\left( q^* (a_0, \alpha_0 ,1 ) - q_1 \right) 
}\int_{B_{R}}V\left( \left| x\right| \right) \left| u\right| ^{p}dx\right) ^{
\frac{p-1}{p}}
\leq  CR^{   \frac{N+ a_0 -p}{p} \left( q^* (a_0, \alpha_0 ,1 ) - q_1 \right)  }.
$$

\noindent So, in any case, we deduce $\mathcal{S}_{0}\left( q_{1},R\right) \leq
CR^{\delta }$ for some $\delta =\delta \left( N,p,\alpha _{0},\beta
_{0},q_{1}\right) >0$ and this concludes the proof.
\end{proof}

\proof[Proof of Theorem \ref{THM1}]
Assume the hypotheses of the theorem and let $u \in X$ be
such that $\left\| u\right\| =1$. Let $R> R_{2}$. We
will denote by $C$ any positive constant which does not depend on $u$ and $R$.

By the pointwise estimates of Lemma \ref{lemdens2} and the fact that 
\[
\esssup_{x\in B_{R}^{c}}\frac{K\left( \left| x\right| \right) }{%
\left| x\right| ^{\alpha _{\infty }}V\left( \left| x\right| \right) ^{\beta
_{\infty }}}\leq \esssup_{r>R_{2}}\frac{K\left( r\right) }{%
r^{\alpha _{\infty }}V\left( r\right) ^{\beta _{\infty }}}<+\infty , 
\]
we can apply Lemma \ref{Lem(Omega2)} with $R_0= R_2$, $\alpha
=\alpha _{\infty }$, $\beta =\beta _{\infty }$, $m=M \left\|u\right\|
=M $ and $\nu =\frac{N-p+ a_{\infty}}{p}$. 
\par \noindent If $0\leq \beta _{\infty }\leq 1/p$ we get 

\begin{eqnarray*}
\int_{B_{R}^{c}}K\left( \left| x\right| \right) \left| u\right|
^{q_{2}} dx &\leq &C\left( \int_{B_{R}^{c}}\left| x\right|
^{\frac{\alpha _{\infty }- \nu \left( q_{2}-1\right) }{N(p-1)-a_{\infty}+p\left(
1-p\beta _{\infty } +a_{\infty} \beta_{\infty}\right) }pN}dx\right) ^{\frac{N(p-1)-a_{\infty}+p\left( 1-p\beta _{\infty
}+a_{\infty} \beta_{\infty}\right) }{pN}}. 
\end{eqnarray*}

\noindent Notice that we have

$$\frac{\alpha _{\infty }- \nu \left( q_{2}-1\right) }{N(p-1)-a_{\infty}+p\left(
1-p\beta _{\infty } +a_{\infty} \beta_{\infty}\right) } \, pN  +N =$$
$$= \frac{N \left( N+ a_{\infty} -p \right) }{N(p-1)-a_{\infty}+p\left(
1-p\beta _{\infty } +a_{\infty} \beta_{\infty}\right) }   \left[ q^* ( a_{\infty}, \alpha_{\infty}, \beta_{\infty}) -q_2 \right] <0 ,$$

\noindent thanks to the hypotheses. Hence we get

$$\int_{B_{R}^{c}}K\left( \left| x\right| \right) \left| u\right|
^{q_{2}} dx \leq $$
$$\leq C \left( \int_{R}^{+\infty} r^{\frac{N \left( N+ a_{\infty} -p \right) }{N(p-1)-a_{\infty}+p\left(
1-p\beta _{\infty } +a_{\infty} \beta_{\infty}\right) }   \left[ q^* ( a_{\infty}, \alpha_{\infty}, \beta_{\infty}) -q_2 \right]-1}dr \right)^{\frac{N(p-1)-
a_{\infty}+p\left( 1-p\beta _{\infty
}+a_{\infty} \beta_{\infty}\right) }{pN}}= C R^{\delta}$$

\noindent for some $\delta <0$.
On the other hand, if $1/p<\beta _{\infty }<1$ we have 
\begin{eqnarray*}
\int_{B_{R}^{c}}K\left( \left| x\right| \right) \left| u\right|
^{q_{2}} dx &\leq &C\left( \int_{B_{R}^{c}}\left| x\right|
^{\frac{\alpha _{\infty }-\nu \left( q_{2}-p\beta _{\infty }\right) 
}{1-\beta _{\infty }}}dx\right) ^{1-\beta _{\infty }} \\
&\leq &C\left( \int_{R}^{+\infty }r^{\frac{\alpha _{\infty }-\nu
\left( q_{2}-p\beta _{\infty }\right) }{1-\beta _{\infty }}+N-1}dr\right)
^{1-\beta _{\infty }},
\end{eqnarray*}

where

$$ \frac{\alpha _{\infty }-\nu
\left( q_{2}-p\beta _{\infty }\right) }{1-\beta_{\infty }}+N=\frac{N+a_{\infty} -p}{p(1-\beta{\infty})}\left( q^* (a_{\infty},\alpha_{\infty}, \beta_{\infty})
-q_2 \right) <0.
$$

\noindent Hence we get

$$\int_{B_{R}^{c}}K\left( \left| x\right| \right) \left| u\right|
^{q_{2}} dx \leq C R^{\frac{N+a_{\infty} -p}{p(1-\beta{\infty})}\left( q^* (a_{\infty},\alpha_{\infty}, \beta_{\infty})
-q_2 \right)}.
$$

\noindent Finally, if $\beta _{\infty }=1$, we obtain 
$$
\int_{B_{R}^{c}}K\left( \left| x\right| \right) \left| u\right|^{q_{2}} dx 
\leq C\left( \int_{B_{R}^{c}}\left| x\right|^{\frac{p}{p-1}(\left( \alpha _{\infty }-\nu \left( q_{2}-p\right) \right) }V\left( \left|x\right| \right) \left| u\right|^{p}dx\right)^{\frac{p-1}{p}} ,
$$

\noindent where

$$\alpha_{\infty} -\nu (q_2 -p) = \frac{N+ a_{\infty} -p}{p} \left( q^* (a_{\infty}, \alpha_{\infty}, \beta_{\infty})-q_2 \right) <0 . $$

\noindent Hence 

$$\int_{B_{R}^{c}}  K\left( \left| x\right| \right) \left| u\right|
^{q_{2}} dx \leq C R^{\frac{N+ a_{\infty} -p}{p} \left( q^* (a_{\infty}, \alpha_{\infty}, \beta_{\infty})-q_2 \right)}
\left(\int_{B_{R}^{c}} V(|x|) |u(x)|^p dx \right)^{\frac{p-1}{p}}\leq$$

$$\leq C R^{\frac{N+ a_{\infty} -p}{p} \left( q^* (a_{\infty}, \alpha_{\infty}, \beta_{\infty})-q_2 \right)}.$$

\noindent So, in any case, we get $\mathcal{S}_{\infty }\left(
q_{2},R\right) \leq CR^{\delta }$ for some $\delta =\delta ( N,p,\alpha_{\infty },\beta _{\infty },q_{2}) <0$, 
which completes the proof.
\endproof


\section{Existence of solutions \label{SEC: ex}}

Let $N\geq 3$ and $1<p<N$. In this section we apply our embedding results to get existence of 
radial weak solutions to the equation 
\begin{equation}
-\mathrm{div}\left(A(|x| |\nabla u|^{p-2} \nabla u\right) u+V\left( \left| x\right| \right) |u|^{p-2}u= K(|x|) f(u) \quad \text{in }\mathbb{R}^{N},  \label{EQg}
\end{equation}
i.e., functions $u\in X$ such that 
\begin{equation}
\int_{\mathbb{R}^{N}} A(|x|) |\nabla u|^{p-2}\nabla u\cdot \nabla h\,dx+\int_{
\mathbb{R}^{N}}V\left( \left| x\right| \right) |u|^{p-2}uh\,dx=\int_{
\mathbb{R}^{N}}K(|x|) f(u) h\,dx\quad 
\forall h\in X,  \label{weak solution}
\end{equation}
where $A$, $V$ and $K$ are potentials satisfying $\left( \mathbf{A}\right) $, $\left( \mathbf{V}\right) $ and $\left( \mathbf{K}\right) $, 
and $X$ and is the Banach spaces defined in Section 2. 

\begin{rem}
We focus on super $p$-linear nonlinearities $f$ just for simplicity, but our compactness results also allow to treat the case of sub $p$-linear $f$'s. Moreover, multiplicity results can also be obtained. We leave the details to interested reader, 
which we refer to \cite{BGR_II,BGR_p} for similar results and related arguments.
\end{rem}

As concerns our hypotheses on the
nonlinearity, we require that $f: {\mathbb{R}} \rightarrow {\mathbb{R}}$ is a continuous function, we set $F(t)= \int_{0}^{t} f(s)ds $, and we assume the following conditions:

\begin{itemize}

\item[$\left( f_{1}\right) $]  $\exists \theta >p$ such that $0\leq \theta
F\left( t\right) \leq f\left( t\right) t$ for all $t\in \mathbb{R}$;

\item[$\left( f_{2}\right) $]  $\exists t_{0}>0$ such that $F\left(t_{0}\right) >0$;

\item[$\left( f_{q_{1},q_{2}}\right) $]  
$\left|f\left( t\right) \right| \leq (\text{const.})\min \left\{ \left| t\right|
^{q_{1}-1},\left| t\right| ^{q_{2}-1}\right\} $ for all $t\in \mathbb{R}$.


\end{itemize}

\noindent We notice that these hypotheses imply $q_1, q_2 \geq \theta $. Also we observe that, if $q_{1}\neq q_{2}$, the double-power growth
condition $\left( f_{q_{1},q_{2}}\right) $ is more stringent than the more
usual single-power one, since it implies $|f(t)| \leq (\text{const.})|t|^{q-1}$
for $q=q_{1}$, $q=q_{2}$ and every $q$
in between. On the other hand, we will never require $q_{1}\neq q_{2}$ in $
\left( f_{q_{1},q_{2}}\right) $, so that our results will also concern
single-power nonlinearities as long as we can take $q_{1}=q_{2}$ (cf. Example \ref{ex2} below).
\bigskip

We set
$$I\left( u\right) :=\frac{1}{p} \left\| u\right\| ^{p}-\int_{\mathbb{R}
^{N}}  K(|x|) F\left( u\right) dx= $$
$$=\frac{1}{p} \int_{\mathbb{R}^{N}} A(|x|) |\nabla u|^p dx + \frac{1}{p} \int_{\mathbb{R}^{N}} V(|x|) |u|^p dx - 
\int_{\mathbb{R}^{N}}  K(|x|) F\left( u\right) dx.
$$

\noindent From the continuous embedding result of Theorem \ref{THM(cpt)} and the
results of \cite{BPR} about Nemytski\u{\i} operators on the sum of Lebesgue
spaces, we have that $I$ is a $C^{1}$ functional on $X$
provided that there exist $q_{1},q_{2}>1$ such that $\left(
f_{q_{1},q_{2}}\right) $ and $\left( \mathcal{S}_{q_{1},q_{2}}^{\prime
}\right) $ hold. In this case, the Fr\'{e}chet derivative of $I$ at any $
u\in X$ is given by 
\begin{equation}
I^{\prime }\left( u\right) h=\int_{\mathbb{R}^{N}}A(|x|)\left( |\nabla u|^{p-2}\nabla
u\cdot \nabla h+V\left( \left| x\right| \right) |u|^{p-2}uh\right) dx-\int_{
\mathbb{R}^{N}}K(|x|) f\left( u\right) h\,dx
\label{PROP:diff: I'(u)h=}
\end{equation}
for all $h\in X$, and therefore the critical points of $I:X\rightarrow \mathbb{R}$ satisfy (\ref{weak solution}).

Our existence result is the following.

\begin{thm}
\label{THM:ex}Assume 
that there
exist $q_{1},q_{2}>p$ such that $\left( f_{q_{1},q_{2}}\right) $ and $\left( 
\mathcal{S}_{q_{1},q_{2}}^{\prime \prime }\right) $ hold.
\noindent Then the functional $I:X\rightarrow \mathbb{R}$ has a nonnegative
critical point $u\neq 0$.
\end{thm}

\begin{rem}
\label{RMK:thm:ex} In Theorem \ref{THM:ex}, the assumptions on $f$ need only to hold for $t\geq
0$. Indeed, all the hypotheses of the theorem still hold true if we replace $
f\left( t\right) $ with $\chi _{\mathbb{R}_{+}}\left( t\right) f\left(
t\right) $ ($\chi _{\mathbb{R}_{+}}$ is the characteristic function of $\mathbb{R
}_{+}$) and this can be done without restriction since the theorem concerns 
nonnegative critical points.

\end{rem}

The above result relies on assumption $\left( 
\mathcal{S}_{q_{1},q_{2}}^{\prime \prime }\right) $, which is quite abstract
but can be granted in concrete cases through Theorems \ref{THM0} and \ref{THM1}, which ensure such assumption for suitable ranges of
exponents $q_{1}$ and $q_{2}$ by explicit conditions on the potentials. 
As concerns examples of nonlinearities satisfying the hypotheses of Theorem \ref{THM:ex}, the simplest 
$f\in C\left(\mathbb{R};\mathbb{R}\right) $ such that $\left( f_{q_{1},q_{2}}\right) $ holds is 
\[f\left( t\right) =\min \left\{ \left| t\right| ^{q_{1}-2}t,\left| t\right|
^{q_{2}-2}t\right\} ,
\]
which also ensures $\left( f_{1}\right) $ if $
q_{1},q_{2}>p$ (with $\theta =\min \left\{ q_{1},q_{2}\right\} $). Another model example is 
\[
f\left( t\right) =\frac{\left| t\right| ^{q_{2}-2}t}{1+\left| t\right|
^{q_{2}-q_{1}}}\quad \text{with }1<q_{1}\leq q_{2},
\]
which ensures $\left( f_{1}\right) $ if $q_{1}>p$ (with $\theta =q_{1}$). Note that, in
both these cases, also $\left( f_{2}\right) $ holds true. Moreover, both of these functions $f$
become $f\left( t\right) =\left| t\right| ^{q-2}t$ if $q_{1}=q_{2}=q$. 

We now prove Theorem \ref{THM:ex}, starting with some lemmas.

\begin{lem}
\label{LEM:MPgeom}Assume the hypotheses of Theorem \ref{THM:ex}. Then there exist three constants $c_{1},c_{2}>0$ such that 
\begin{equation}
I\left( u\right) \geq \frac{1}{p}\left\| u\right\| ^{p}-c_{1}\left\|
u\right\| ^{q_{1}}-c_{2}\left\| u\right\| ^{q_{2}}
\qquad \text{for all }u\in X.
  \label{LEM:MPgeom: th}
\end{equation}
\end{lem}

\proof
Let $i\in \left\{ 1,2\right\} $. By the monotonicity of $\mathcal{S}_{0}$
and $\mathcal{S}_{\infty }$, it is not restrictive to assume $R_{1}<R_{2}$
in hypothesis $\left( \mathcal{S}_{q_{1},q_{2}}^{\prime }\right) $. Then, by
lemmas \ref{Lem(corone)} and \ref{A6}, there exists a constant $c_{R_{1},R_{2}} >0$ 
such that for all $u\in X$ we have

\[
\int_{B_{R_{2}}\setminus B_{R_{1}}}K\left( \left| x\right| \right) \left|
u\right| ^{q_{1}}dx\leq C_{R_{1},R_{2}}\left\| u\right\|
^{q_{1}}. 
\]

Therefore, by the hypotheses on $f$ and the definitions of $\mathcal{S}_{0}$ and $
\mathcal{S}_{\infty }$, we obtain 
\begin{eqnarray}
&&\left| \int_{\mathbb{R}^{N}}K\left( \left| x\right| \right) F\left( u\right) dx\right| \leq
\nonumber
\\
&\leq & C \int_{\mathbb{R}^{N}}K\left( \left| x\right| \right) \min
\left\{ \left| u\right| ^{q_{1}},\left| u\right| ^{q_{2}}\right\}
dx \nonumber \\
&\leq & C \left( \int_{B_{R_{1}}}K\left( \left| x\right| \right)
\left| u\right| ^{q_{1}}dx+\int_{B_{R_{2}}^{c}}K\left( \left| x\right|
\right) \left| u\right| ^{q_{2}}dx+\int_{B_{R_{2}}\setminus
B_{R_{1}}}K\left( \left| x\right| \right) \left| u\right| ^{q_{1}}dx\right)
 \nonumber \\
&\leq & C \left( \left\| u\right\| ^{q_{1}}\mathcal{S}_{0}\left(
q_{1},R_{1}\right) +\left\| u\right\| ^{q_{2}}\mathcal{S}_{\infty }\left(
q_{2},R_{2}\right) +C_{R_{1},R_{2}}\left\| u\right\|
^{q_{1}}\right)   
\\
&=&C_{1}\left\| u\right\| ^{q_{1}}+C_{2}\left\| u\right\|
^{q_{2}},  \nonumber
\end{eqnarray}
where the constants $c_{1}$ and $c_{2}$ are independent of 
$u$. This yields (\ref{LEM:MPgeom: th}). 

\endproof

\begin{lem}
\label{LEM:PS}Under the assumptions of Theorem \ref{THM:ex}, the functional $
I:X\rightarrow \mathbb{R}$ satisfies the Palais-Smale condition.
\end{lem}

\proof

Let $\left\{ u_{n}\right\} $ be a sequence in $X$ such that $\left\{
I\left( u_{n}\right) \right\} $ is bounded and $I^{\prime }\left(
u_{n}\right) \rightarrow 0$ in $X^{\prime }$. Hence 
\[
\frac{1}{p}\left\| u_{n}\right\| ^{p}-\int_{\mathbb{R}^{N}}K(|x|) F\left( u_{n}\right) dx=O\left( 1\right) \quad \text{and}\quad \left\|
u_{n}\right\| ^{p}-\int_{\mathbb{R}^{N}}K(|x|) f\left( u_{n}\right)
u_{n}dx=o\left( 1\right) \left\| u_{n}\right\| . 
\]
As $f$ satisfies $\left( f_{1}\right) $, we get 
\[
\frac{1}{p}\left\| u_{n}\right\| ^{p}+O\left( 1\right) =\int_{\mathbb{R} 
^{N}} K(|x|)F\left( u_{n}\right)  dx\leq \frac{1}{\theta }\int_{%
\mathbb{R}^{N}}K(|x|)f\left( u_{n}\right) u_{n}dx=\frac{1}{\theta }%
\left\| u_{n}\right\| ^{p}+o\left( 1\right) \left\| u_{n}\right\| , 
\]
which implies that $\left\{ \left\| u_{n}\right\| \right\} $ is bounded
since $\theta >p$. 
Now, thanks to assumption $\left( \mathcal{S}_{q_{1},q_{2}}^{\prime \prime
}\right) $, we apply Theorem \ref{THM(cpt)} to deduce the existence of $u\in
X$ such that (up to a subsequence) $u_{n}\rightharpoonup u$ in $X$
and $u_{n}\rightarrow u$ in $L_{K}^{q_{1}}+L_{K}^{q_{2}}$. Setting 
\[
I_{1}\left( u\right) :=\frac{1}{p}\left\| u\right\| ^{p}\quad \text{and}
\quad I_{2}\left( u\right) :=I_{1}\left( u\right) -I\left( u\right) 
\]
for brevity, we have that $I_{2}$ is of class $C^{1}$ on $%
L_{K}^{q_{1}}+L_{K}^{q_{2}}$ by \cite[Proposition 3.8]{BPR} and therefore we
get $\left\| u_{n}\right\| ^{p}=I^{\prime }\left( u_{n}\right)
u_{n}+I_{2}^{\prime }\left( u_{n}\right) u_{n}=I_{2}^{\prime }\left(
u\right) u+o\left( 1\right) $. Hence $\lim_{n\rightarrow \infty }\left\|
u_{n}\right\| $ exists and one has $\left\| u\right\| ^{p}\leq
\lim_{n\rightarrow \infty }\left\| u_{n}\right\| ^{p}$ by weak lower
semicontinuity. Moreover, the convexity of $I_{1}:X\rightarrow \mathbb{R}$
implies 
\[
I_{1}\left( u\right) -I_{1}\left( u_{n}\right) \geq I_{1}^{\prime }\left(
u_{n}\right) \left( u-u_{n}\right) =I^{\prime }\left( u_{n}\right) \left(
u-u_{n}\right) +I_{2}^{\prime }\left( u_{n}\right) \left( u-u_{n}\right)
=o\left( 1\right) 
\]
and thus 
\[
\frac{1}{p}\left\| u\right\| ^{p}=I_{1}\left( u\right) \geq
\lim_{n\rightarrow \infty }I_{1}\left( u_{n}\right) =\frac{1}{p}%
\lim_{n\rightarrow \infty }\left\| u_{n}\right\| ^{p}. 
\]
So $\left\| u_{n}\right\| \rightarrow \left\| u\right\| $ and one concludes
that $u_{n}\rightarrow u$ in $X$ by the uniform convexity of the norm.%
\endproof


\proof[Proof of Theorem \ref{THM:ex}.] 
Assume the hypotheses of the theorem, together with the following non-restrictive additional condition:
$f(t)=0$ for $t<0$.
We want to apply the Mountain-Pass Theorem. To this end, from (\ref{LEM:MPgeom: th}) of Lemma \ref
{LEM:MPgeom} we deduce that, since $q_{1},q_{2}>p$, there
exists $\rho >0$ such that 
\begin{equation}
\inf_{u\in X,\,\left\| u\right\| =\rho }I\left( u\right) >0=I\left(
0\right) .  \label{mp-geom}
\end{equation}
Therefore, taking into account Lemma \ref{LEM:PS}, we need only to check
that $\exists \bar{u}\in X$ such that $\left\| \bar{u}\right\| >\rho $
and $I\left( \bar{u}\right) <0$. To this end, from assumption $\left( f_{1}\right) $ and $
\left( f_{2}\right) $ we infer that 
\[
F\left( t\right) \geq \frac{F\left( t_{0}\right) }{t_{0}^{\theta }}
t^{\theta }\text{ for all }t\geq t_{0}. 
\]
We then fix a non negative
function $u_{0}\in C_{c}^{\infty }( \mathbb{R}^{N} \backslash \{0\}  )$ such that the set $\{x\in \mathbb{R}^{N}:u_{0}\left(
x\right) \geq t_{0}\}$ has positive Lebesgue measure. Hence, since $\left( f_{1}\right) $ and $\left( f_{2}\right) $ ensure
that $F(t)\geq 0$ for all $t$ and $F\left( t_{0}\right) >0$, for every $\lambda >1$ we get 
\begin{eqnarray*}
\int_{\mathbb{R}^{N}}K(|x|) F\left( \lambda u_{0}\right) dx &\geq
&\int_{\left\{ \lambda u_{0}\geq t_{0}\right\} }K(|x|) F\left(\lambda u_{0}\right) dx\geq \frac{\lambda ^{\theta }}{t_{0}^{\theta }}
\int_{\left\{ \lambda u_{0}\geq t_{0}\right\} }K(|x|) F\left( t_{0}\right) u_{0}^{\theta }dx \\
&\geq &\frac{\lambda ^{\theta }}{t_{0}^{\theta }}\int_{\left\{ u_{0}\geq
t_{0}\right\} }K(|x|) F\left( t_{0}\right) u_{0}^{\theta }dx\geq
\lambda ^{\theta }\int_{\left\{ u_{0}\geq t_{0}\right\} }K(|x|) F\left( t_{0}\right) dx>0.
\end{eqnarray*}
Since $\theta >p$, this gives 
\[
\lim_{\lambda \rightarrow +\infty }I\left( \lambda u_{0}\right) \leq
\lim_{\lambda \rightarrow +\infty }\left( \frac{\lambda ^{p}}{p}\left\|
u_{0}\right\| ^{p}-\lambda ^{\theta }\int_{\left\{ u_{0}\geq t_{0}\right\}
} K(|x|) F\left( t_{0}\right) dx\right) =-\infty . 
\]
As a conclusion, we can take $\bar{u}=\lambda u_{0}$ with $\lambda $
sufficiently large and the Mountain-Pass Theorem provides the existence of a
nonzero critical point $u\in X$ for $I$. Since the additional assumption $f(t)=0$ for $t<0$
implies $
I^{\prime }\left( u\right) u_{-}=-\left\| u_{-}\right\| ^{p}$ (where $
u_{-}\in X$ is the negative part of $u$), one concludes that $u_{-}=0$,
i.e., $u$ is nonnegative.
\endproof

\section{Examples}\label{SEC:EX}

\noindent In this section we give some examples that might help to understand what is new in our results. We will make a comparison, in concrete cases, between our results and those of \cite{Su-Wang}. In that paper the authors prove some compactness theorems which are used to prove existence results for equation (\ref{EQ}), where $f$ is a power or a sum of powers. We will show some cases where the results of \cite{Su-Wang} do not apply, while our results give existence of solutions. In all our example we look for a nonlinearity defined by $f(t)= \min \{ t^{q_1 -1} , t^{q_2 -1} \}$, and we will see how to choose $p<q_1 \leq q_2$ in such a way to get existence results for problem (\ref{EQ}).

\begin{exa} \label{ex1}
Let $A, V, K$ be as follows:

$$A(r)= \min \{ r^{1/2} , r^{3/2} \}, \quad V(r)= \min \left\{ 1, r^{-3/2}  \right\}, \quad K(r)= \max \{ r^{1/2}, r^{3/2}  \}.$$

\noindent Assume $3/2 <p \leq 2$. We first show that in this case the results of \cite{Su-Wang} do not apply. The embedding theorems of that paper are Theorems 3.2, 3.3 and 3.4. If we compute the coefficients $q^*$ and $q_*$ of \cite{Su-Wang}, we easily obtain $ q_{*}= \frac{p(2N+3)}{2N+1-2p}$, $ q^{*}= \frac{p(2N+1)}{2N+3-2p}$ and this, together with $p\leq 2$, implies 
$q^* < q_*$, so that Theorem 3.2 of \cite{Su-Wang} cannot be applied, because it needs $ q_* < q^*$. One easily verifies that also the hypotheses of Theorems 3.3 and 3.4 of \cite{Su-Wang} are not satisfied. To apply our results, we set $\beta_0 = \beta_{\infty}=0$, $\alpha_0 = 1/2$, $\alpha_{\infty}=3/2$, $a_0= 3/2$, $a_{\infty}=1/2$. Note that condition $a_0 , a_{\infty} \in (p-N,p]$
is satisfied. We apply Theorems \ref{THM0} and \ref{THM1}, and we compute 

$$q^{*} (a_0 , \alpha_0, \beta_0 ) = \frac{p(2N+1)}{2N+3-2p}, \quad  q_{*} (a_{\infty} , \alpha_{\infty}, \beta_{\infty} )=\frac{p(2N+3)}{2N+1-2p}.$$

\noindent Notice that these are the same value obtained above, following \cite{Su-Wang}. Notice also that $q^{*} (a_0 , \alpha_0, \beta_0 ) >p$ is equivalent to $p> 1 $. Applying our results, we deduce that if we take $q_1 ,q_2$ such that

$$p<q_1 < \frac{p(2N+1)}{2N+3-2p}< \frac{p(2N+3)}{2N+1-2p}<q_2 ,$$

\noindent then $\left(\mathcal{S}_{q_{1},q_{2}}^{\prime \prime }\right) $ 
holds and we get an existence result for the equation (\ref{EQ}) with any nonlinearity satisfying 
$\left( f_{q_{1},q_{2}}\right) $ .
\end{exa}

\begin{exa} \label{ex2}
Assume $N\geq 4$ and choose the functions $A, V, K$ as follows:

$$A(r)=  \max \left\{ r^{-2}, r^{-1} \right\} , \quad V(r)= e^{2r} , \quad K(r)= e^r  .$$

\noindent In this case the results of \cite{Su-Wang} do not apply because of the exponential growth of the potential $K$. Assume $1<p<N-2$. In order to apply Theorems \ref{THM0} and \ref{THM1}, we can choose $a_0= -2 $, $a_{\infty} =-1$, 
$\beta_0 =\alpha_{0}= \alpha_{\infty}=0$, $\beta_{\infty}=1/2$. 
Notice that condition $a_0 , a_{\infty} \in (p-N,p]$ is satisfied. 
We get $$q^* (a_0, \alpha_0 , \beta_0 ) = \frac{pN}{N-p-2}\quad \text{and}\quad 
q^* (a_{\infty}, \alpha_{\infty} , \beta_{\infty} ) =\frac{pN}{N-p-1},$$ where we have
$
p<\frac{pN}{N-p-1}<\frac{pN}{N-p-2}.
$
Then $\left(\mathcal{S}_{q_{1},q_{2}}^{\prime \prime }\right) $ 
holds and we get an existence result for the equation (\ref{EQ}) with any nonlinearity satisfying 
$\left( f_{q_{1},q_{2}}\right) $ 
provided that $$p<q_1 < \frac{pN}{N-p-2}\quad \text{and}\quad  q_2 > \frac{pN}{N-p-1}.$$ In particular we can take a power nonlinearity $f(t)= t^{q-1}$ for $q \in \left( \frac{pN}{N-p-1}, \frac{pN}{N-p-2} \right)$. 

\end{exa}

\end{document}